\documentclass[reqno]{amsart}
\usepackage{amsmath}
\usepackage{amsfonts}
\usepackage{amssymb}
\usepackage{amsthm}

\newcommand{\ot}{\otimes}\newcommand{\vt}{\overline{\otimes}}
\newcommand{\ft}{\overline{\otimes}_{\cl{F}}}
\newcommand{\cl}{\mathcal}
\newcommand{\sub}{\subseteq}
\newcommand{\bfc}{\textit}
\newcommand{\ap}{\infty}
\newcommand{\wsp}{\overline{\mathrm{span}}^{\mathrm{w}^{*}}}
\newcommand{\wh}{\widehat}
\newcommand{\wt}{\widetilde}
\newcommand{\BLT}{B(L^{2}(G))}
\newcommand{\LO}{L^{1}(G)}
\newcommand{\LT}{L^{2}(G)}
\newcommand{\LI}{L^{\infty}(G)}
\newcommand{\CI}{\mathbb{C}1_{H}}
\newcommand{\CIK}{\mathbb{C}1_{K}}
\newcommand{\la}{\langle}
\newcommand{\ra}{\rangle}
\newcommand{\fcr}{\rtimes^{\cl{F}}}
\newcommand{\scr}{\overline{\rtimes}}
\newcommand{\dfcr}{\ltimes^{\cl{F}}}
\newcommand{\dscr}{\overline{\ltimes}}
\newcommand{\Id}{\mathrm{id}}
\newcommand{\Ad}{\mathrm{Ad}}
\newcommand{\B}{\mathrm{Bim}(J^{\perp})}
\newcommand{\Sa}{N(J)}
\newcommand{\lt}{\ltimes}

\newtheorem{thm}{Theorem}[section]
\newtheorem{lem}[thm]{Lemma}
\newtheorem{pro}[thm]{Proposition}
\newtheorem{cor}[thm]{Corollary}
\theoremstyle{definition}
\newtheorem{defin}[thm]{Definition}
\newtheorem{remark}[thm]{Remark}

\begin{document} 
	
	\begin{abstract} 
		Let $ G $ be a locally compact group. We study the categories of $ L^{\infty}(G) $-comodules and $ L(G) $-comodules in the setting of dual operator spaces and the associated crossed products. It is proved that every $ L^{\infty}(G) $-comodule is non-degenerate and saturated, whereas every $ L(G) $-comodule is non-degenerate if and only if every $ L(G) $-comodule is saturated if and only if $ G $ has the approximation property in the sense of Haagerup and Kraus \cite{HK}. This allows us to extend known results from the duality theory of crossed products of von Neumann algebras (such as Takesaki-duality and the Digernes-Takesaki theorem) to the recent theory of crossed products of dual operator spaces. As applications, we obtain a characterization of groups with the approximation property in terms of the related crossed products improving a recent result of Crann and Neufang \cite{CN} and we generalize a theorem of Anoussis, Katavolos and Todorov \cite{AKT1} providing a less technical proof of it. Furthermore, this approach provides an answer to a question raised by the authors in \cite{AKT1}.
	\end{abstract}	

	\title[crossed products and the AP]{Crossed products of dual operator spaces and a characterization of groups with the approximation property} 
	\author[D. Andreou]{Dimitrios Andreou}
	\address{Department of Mathematics, National and Kapodistrian University of Athens, Athens 157 84, Greece}
	\email{dimandreou@math.uoa.gr}
	\maketitle 	
		
\section{Introduction}\label{sec1}	

The study of possible ways of extending the classical crossed product construction for von Neumann algebras and related results (e.g. the Takesaki-type duality and the Digernes-Takesaki theorem) to the setting of more general operator spaces has recently received special attention by many different authors. This is because crossed product-type constructions for actions of locally compact groups on operator spaces arise naturally as necessary tools at the meeting point of operator space theory and abstract harmonic analysis and more precisely in the study of harmonic operators, non-commutative Poisson boundaries and groups with the approximation property (AP) in the sense of Haagerup and Kraus \cite{HK}. For example, see \cite{AKT}, \cite{CN}, \cite{Ha1}, \cite{Iz}, \cite{JN}, \cite{KNR}, \cite{SS}, \cite{UZ} and the references therein.

In this article we attempt a more systematic study of crossed products associated with actions of locally compact groups (and `group duals') on dual operator spaces. The main purpose is to investigate under which conditions the fundamental properties of crossed products of von Neumann algebras still hold in the category of dual operator spaces. It turns out that some known results can be easily generalized in this more general setting whereas others require non-trivial further assumptions (e.g. the acting group to have the AP).

We follow the Hopf-von Neumann algebra approach since it is the natural way to formulate the duality theory for locally compact groups in the non-abelian case. This idea has already been applied in a very fruitful way to the classical duality theory of crossed products of von Neumann algebras (see e.g. \cite{La}, \cite{La2}, \cite{Nak}, \cite{NaTa}, \cite{SVZ1}, \cite{SVZ2} and \cite{Tak2}) as well as to the study of crossed products of operator spaces (see \cite{CN}, \cite{Ha1} and \cite{SS}). More precisely, if $ G $ is a locally compact group (not necessarily abelian), then $ G $ and its  `dual' should respectively be replaced by the Hopf-von Neumann algebras $ (\LI,\alpha_{G}) $ and $ (L(G),\delta_{G}) $. Furthermore, group actions should be considered as $ \LI $-comodules and actions of group duals should be thought of as $ L(G) $-comodules (see section \ref{sec2} for the precise definitions).

There are two natural ways to define the crossed product for an $ \LI $-comodule $ (X,\alpha) $, which both coincide with the usual von Neumann algebraic crossed product if $ X $ is a von Neumann algebra and $ \alpha $ corresponds to a $ G $-action by unital normal *-automorphisms, namely the \bfc{spatial crossed product} $ X\scr_{\alpha}G $ and the \bfc{Fubini crossed product} $ X\fcr_{\alpha}G $ (see Definitions \ref{deffcr}, \ref{def4.2} and \ref{def4.4}). The definition of $ X\scr_{\alpha}G $ is suggested by the covariance relations which allow us to describe the crossed product of a von Neumann algebra as the normal $ L(G) $-bimodule generated by the image of the corresponding comodule action (see the comment right after Definition \ref{def4.4}). On the other hand, the definition of $ X\fcr_{\alpha}G $ is suggested by the Digernes-Takesaki theorem which characterizes the crossed product of a von Neumann algebra as the space of fixed points of a certain action (see the comments before Definition \ref{deffcr}).

Similarly, for an action of the dual object of a locally compact group $ G $ on a dual operator space, which in the non-abelian case is an $ L(G) $-comodule $ (Y,\delta) $, the spatial and the Fubini crossed products $ Y\dscr_{\delta}G $ and $ Y\dfcr_{\delta}G $ respectively are defined in an analogous way (see Definitions \ref{dtil}, \ref{def4.2i} and \ref{def2.11}).

An interesting class of dual operator spaces arising as Fubini crossed products of w*-closed subspaces of $ \LI $ by $ G $ acting by left translations is the class of jointly harmonic operators introduced by Anoussis, Katavolos and Todorov \cite{AKT} (see \cite[Theorem 7.1]{AKT} and \cite[Proposition 5.1]{DA}). Jointly harmonic operators naturally generalize the notion of non-commutative Poisson boundary of a probability measure on $ G $ initially introduced by Izumi \cite{Iz} and further studied by Jaworski and Neufang \cite{JN} and Kalantar, Neufang and Ruan \cite{KNR}. The main result concerning the non-commutative Poisson boundary of a single probability measure $ \mu $ on $ G $ is that it can be realized as the crossed product of the space of $ \mu $-harmonic functions which admits a certain von Neumann algebra structure. The analogue of this in the case of the jointly harmonic operators with respect to an arbitrary subset of the measure algebra $ M(G) $ is equivalent to the equality $ X\scr_{\alpha}G=X\fcr_{\alpha}G $ where $ X $ is a certain translation-invariant subspace of $ \LI $, not necessarily a von Neumann subalgebra, and $ \alpha $ is the canonical action of $ G $ by left translations (again by \cite[Theorem 7.1]{AKT} and \cite[Proposition 5.1]{DA}). This was shown to be true in \cite{AKT} when $ G $ is weakly amenable discrete, or abelian, or compact, but it remains open whether it holds for arbitrary locally compact groups.

On the other hand, Crann and Neufang \cite{CN} generalized the aforementioned result of \cite{AKT} by proving that if $ G $ is a locally compact group with the approximation property (AP) of Haagerup-Kraus, then $ X\scr_{\alpha}G=X\fcr_{\alpha}G $ for any $ \LI $-comodule $ (X,\alpha) $. Note that for a discrete group $ G $, the converse implication is an easy consequence of \cite[Theorem 2.1]{HK} and the observation that if $ \alpha $ is the trivial action on $ X $, then we have $ X\fcr_{\alpha}G=X\ft L(G) $ and $X\scr_{\alpha}G =X\vt L(G) $ (see Remark \ref{rem9}). We will prove that the converse of the above result of Crann and Neufang is true for arbitrary locally compact groups (see Theorem \ref{thm4}), thus providing another characterization of the AP.    

In \cite{SS} Salmi and Skalski proved that $ X\scr_{\alpha}G=X\fcr_{\alpha}G $ when $ X $ is a non-degenerately represented W*-TRO and $ \alpha $ is an action of a locally compact group $ G $ on $ X $ by W*-TRO-automorphisms. Moreover, they extended this result for actions of locally compact quantum groups on W*-TRO's. From this it is easily inferred that $ Y\dscr_{\delta}G = Y\dfcr_{\delta}G $ if $ (Y,\delta) $ is an $ L(G) $-comodule with $ Y $ being a W*-TRO and $ \delta $ being an $ L(G) $-action and also a non-degenerate W*-TRO morphism (since $ L(G) $ has a natural locally compact quantum group structure). We will show a generalization of the latter, namely that $ Y\dscr_{\delta}G = Y\dfcr_{\delta}G $ holds true for any $ L(G) $-comodule $ (Y,\delta) $ (see Theorem \ref{thm1}). This provides an alternative --- and perhaps less technical --- proof of a theorem by Anoussis, Katavolos and Todorov \cite[Theorem 3.2]{AKT1} on a certain class of normal masa bimodules, which arise as crossed products of $ L(G) $-comodules that are not necessarily W*-TRO's (see Proposition \ref{p12}).

Our main results could be briefly summarized in the following four statements:
\begin{itemize}
	\item [(1)] For a locally compact group $ G $ and any $ L(G) $-comodule $ (Y,\delta) $ it holds that \[Y\dscr_{\delta}G=Y\dfcr_{\delta}G.\]
	\item [(2)] For a locally compact group $ G $ and any $ \LI $-comodule $ (X,\alpha) $ the Takesaki-duality holds for both the spatial and the Fubini crossed products, that is \[(X\fcr_{\alpha}G)\dfcr_{\wh{\alpha}}G=(X\scr_{\alpha}G)\dscr_{\wh{\alpha}}G \simeq X\vt\BLT.\]
	\item [(3)] For a locally compact group $ G $, the following are equivalent:
	\begin{itemize}
		\item $ G $  has the approximation property (AP) in the sense of \cite{HK};
		\item For any $ L(G) $-comodule $ (Y,\delta) $, the Takesaki-duality holds for the spatial crossed products, i.e. \[(Y\dscr_{\delta}G)\scr_{\wh{\delta}}G \simeq Y\vt\BLT;\]
		\item For any $ L(G) $-comodule $ (Y,\delta) $, the Takesaki-duality holds for the Fubini crossed products, i.e. \[(Y\dfcr_{\delta}G)\fcr_{\wh{\delta}}G \simeq Y\vt\BLT.\]
	\end{itemize}	
	 \item [(4)] A locally compact group $ G $ has the AP if and only if $ X\scr_{\alpha}G=X\fcr_{\alpha}G $ for any $ \LI $-comodule $ (X,\alpha) $.
\end{itemize}

Statements (1) and (2) directly generalize known facts about crossed products of von Neumann algebras and do not require special assumptions for the group $ G $. On the other hand, their dual versions, that is statements (4) and (3) respectively, are valid if and only if the group $ G $ has the approximation property. This is very interesting since the analogues of (3) and (4) in the von Neumann algebra setting  are true  even for locally compact groups without the approximation property. For example, see \cite[Chapter I, Theorems 2.5 and 2.7]{NaTa} as well  as \cite[Chapter II, Theorem 1.2]{NaTa}.

Our key idea is that all the conditions under question in the aforementioned statements (1) to (4) can be reformulated in terms of the notions of \bfc{saturation} and \bfc{non-degeneracy} for comodules (see section \ref{sec3} for the exact definitions). As a result, the validity of statements (1) to (4) actually reflects two fundamental facts about the Hopf-von Neumann algebras $ (\LI,\alpha_{G}) $ and $ (L(G),\delta_{G}) $ proved in Lemma \ref{c2} and Proposition \ref{p9}, namely the following:

\begin{itemize}
	\item [(a)] For a locally compact group $ G $, every $ \LI $-comodule is non-degenerate and saturated.
	\item [(b)] For a locally compact group $ G $, the following are equivalent:
	\begin{itemize}
		\item Every $ L(G) $-comodule is saturated;
		\item Every $ L(G) $-comodule is non-degenerate;
		\item $ G $ has the approximation property (AP).
	\end{itemize}
\end{itemize}

This paper is organized in six sections including this introduction as section \ref{sec1}. In section \ref{sec2} we begin with presenting the basic concepts and facts concerning tensor products of dual operator spaces, general Hopf-von Neumann algebras and comodules. At the end of section 2 the basic definitions and properties regarding the Hopf-von Neumann algebras $ (\LI,\alpha_{G}) $ and $ (L(G),\delta_{G}) $ are summarized.

In section \ref{sec3}  the notions of saturation and non-degeneracy for general comodules of an arbitrary Hopf-von Neumann algebra are discussed. The results of this section will be applied in the following sections for the Hopf-von Neumann algebras $ (\LI,\alpha_{G}) $ and $ (L(G),\delta_{G}) $.

In section \ref{sec4} we study the spatial crossed product $ X\scr_{\alpha}G $ and the Fubini crossed product $ X\fcr_{\alpha}G $ for an $ \LI $-comodule $ (X,\alpha) $ as well as their dual versions $ Y\dscr_{\delta}G $ and $ Y\dfcr_{\delta}G $ for an $ L(G) $-comodule $ (Y,\delta) $. Interestingly, it turns out that \[Y\dscr_{\delta}G=Y\dfcr_{\delta}G\] for any $ L(G) $-comodule $ (Y,\delta) $ (Theorem \ref{thm1}). This is because the crossed products associated to an $ L(G) $-comodule have a natural $ \LI $-comodule structure and therefore they are non-degenerate and saturated without any further assumption for the group $ G $ (Lemma \ref{c2}).  

On the other hand, if $ (X,\alpha) $ is an $ \LI $-comodule, $ X\scr_{\alpha}G $ and $X\fcr_{\alpha}G$ admit a natural $ L(G) $-comodule structure such that $ X\scr_{\alpha}G $ is non-degenerate and $ X\fcr_{\alpha}G $ is saturated (Corollary \ref{c10}). Furthermore, it is proved that every $ L(G) $-comodule is saturated if and only if every $ L(G) $-comodule is non-degenerate if and only if $ G $ has the AP (see Proposition \ref{p9}). Thus, one would expect that for a group $ G $ without the AP the analogue of Theorem \ref{thm1} for $ \LI $-comodules is not valid. 

This is proved to be the case in section \ref{sec5}. Indeed, it is shown therein that the Takesaki-type duality for the Fubini crossed products is equivalent to the saturation property, while its analogue for the spatial crossed products is equivalent to the non-degeneracy property (see Propositions \ref{p4}, \ref{p15}, \ref{p14} and \ref{p16}). It follows that, for an $ \LI $-comodule $ (X,\alpha) $, the equality $ X\scr_{\alpha}G=X\fcr_{\alpha}G $ holds if and only if $ X\fcr_{\alpha}G $ is non-degenerate or equivalently $ X\scr_{\alpha}G $ is saturated (Theorem \ref{thm2}). Thanks to the above results, we prove that the locally compact group $ G $ has the AP if and only if $ X\scr_{\alpha}G=X\fcr_{\alpha}G $ for any $ \LI $-comodule $ (X,\alpha) $ (i.e. Theorem \ref{thm4}), which improves \cite[Corollary 4.8]{CN}.

In section \ref{sec6} we describe how the $ \LI $-bimodules defined and studied in \cite{AKT1} and \cite{AKT2} can be realized as crossed products of certain $ L(G) $-comodules which are not necessarily von Neumann algebras. This approach yields an alternative proof of one of the main results of \cite{AKT1} (namely \cite[Theorem 3.2]{AKT1}) as a special case of Theorem \ref{thm1}. Also, we answer a question raised by the authors in \cite[Question 4.8]{AKT1} by proving a stronger version of \cite[Lemma 4.5]{AKT1} (see proposition \ref{p7}). These results complement the ones in \cite{DA}, where the $L(G)$-bimodules introduced in \cite{AKT} were identified with crossed products of certain $ \LI $-comodules, thus providing an alternative and unified view of these objects in both the $ \LI $-bimodule and the $L(G)$-bimodule case.

\section{Preliminaries}\label{sec2}	
In this section we list the basic definitions and facts concerning tensor products of dual operator spaces, Hopf-von Neumann algebras and comodules. We follow mainly M. Hamana's terminology (see \cite{Ha1} and \cite{Ha2} for more details) with the difference that in our context all comodules are assumed to be dual operator spaces (concrete in most cases) and comodule actions and morphisms are assumed to be w*-continuous.

\subsection{Tensor products}

Let $ X\sub B(H) $ and $ Y\sub B(K) $ be dual operator spaces, i.e. w*-closed subspaces of $ B(H) $ and $ B(K) $ respectively, where $ H,\ K $ are Hilbert spaces. The \bfc{spatial tensor product} of $ X $ and $ Y $ is the subspace of $ B(H\ot K) $ defined by \[X\vt Y=\wsp\{x\ot y:\ x\in X,\ y\in Y \}, \]
where $ (x\ot y)(h\ot k) =(xh)\ot(yk)$, for $ h\in H ,\ k\in K$.

The \bfc{Fubini tensor product} of $ X $ and $ Y $ is the space:
\begin{align*} 
X\ft Y=\{x\in B(H\ot K) :\  
&(\mathrm{id}_{B(H)}\ot \phi)(x)\in X,\ (\omega\ot \mathrm{id }_{B(K)})(x)\in Y,\\  &\forall\omega\in B(H)_{*},\ \forall\phi\in B(K)_{*} \}.
\end{align*}
Obviously, we have $ X\vt Y\sub X\ft Y $. Furthermore, it is immediate from the definition that for any families $ \{X_{i}\}_{i\in I} $ and $ \{Y_{j}\}_{j\in J} $ of dual operator spaces we have that \[\bigcap_{i,j}(X_{i}\ft Y_{j})=(\bigcap_{i} X_{i} ) \ft ( \bigcap_{j}Y_{j} ).  \]

If $ M $ is an injective von Neumann algebra (in particular, 
of type I) then $ X\vt M= X\ft M $, for every dual operator space $ X $ (see \cite[Theorem 1.9]{Kra}) and this implies that	
\[X\ft Y=(X\vt B(K)) \cap (B(H)\vt Y).\]
   
Also, for any von Neumann algebras $ M$ and $ N $, it holds that $ M\vt N= M\ft N $ (see \cite[Theorem 7.2.4]{ER}).	  

\subsection{Hopf-von Neumann algebras and comodules}

A \bfc{Hopf-von Neumann algebra} is a pair $ (M,\Delta) $, where $ M $ is a von Neumann algebra and $ \Delta\colon M\to M\vt M $ is a normal unital *-monomorphism called the \bfc{comultiplication} of $ M $, such that the coassociativity rule holds:
\[(\Delta\otimes \mathrm{id }_{M})\circ\Delta=(\mathrm{id }_{M}\otimes\Delta)\circ\Delta. \]	

Let $ (M,\Delta) $ be a Hopf-von Neumann algebra. An	$ M $-\bfc{comodule} $ (X,\alpha) $ is a dual operator space $ X $ with a w*-continuous complete isometry $ \alpha\colon X\to X\ft M $ which satisfies
\[(\alpha\otimes \mathrm{id }_{M})\circ\alpha=(\mathrm{id }_{X}\otimes\Delta)\circ\alpha. \]
In this case, we say that $ \alpha $ is an \bfc{action} of $ M $ on $ X $ or an $ M $-\bfc{action} on $ X $.

A w*-closed subspace $ Y $ of $ X $ is called an $ M $-\bfc{subcomodule} of $ X $ if $\alpha(Y)\sub Y\ft M$. In this case we write $ Y\leq X $ and $ Y $ is indeed an $ M $-comodule for the action $ \alpha|_{Y} $. 

An $ M $-\bfc{comodule morphism} between $ M $-comodules $ (X,\alpha) $ and $ (Y,\beta) $ is a w*-w*-continuous complete contraction $ \phi\colon X\to Y $, such that
\[\beta\circ\phi=(\phi\otimes \mathrm{id }_{M})\circ\alpha. \]

An $ M $-comodule morphism is called an $ M $-\bfc{comodule monomorphism} (resp. isomorphism) if it is a complete isometry (resp. surjective complete isometry) and we write $ X\simeq Y $ for isomorphic $ M $-comodules.

If $ X $ is any dual operator space, then the Fubini tensor product $ X\ft M $ becomes an $ M $-comodule, called a \bfc{canonical} $ M $-comodule, with the action \[\mathrm{id }_{X}\otimes\Delta\colon X\ft M\to X\ft M\ft M.\] 

 More generally, for any dual operator space $ X $ and any $ M $-comodule $ (Y,\beta) $, the map $ \Id_{X}\ot\beta $ defines an $ M $-action on the Fubini product $ X\ft Y $.

If $ N $ is a von Neumann algebra, then an $ M $-action $ \pi\colon N\to N\vt M $ on $ N $ that is additionally a normal unital *-monomorphism will be called a W*-$ M $-\bfc{action} on $ N $ and $ (N,\pi) $ will be called a W*-$ M $-\bfc{comodule}. The terms W*-$ M $-\bfc{subcomodule}, W*-$ M $-\bfc{comodule morphism} etc, are defined accordingly. 

\begin{remark}\label{rem1} 
	Let $ (M,\Delta) $ be a Hopf-von Neumann algebra. Every $ M $-comodule $ (X,\alpha) $ is isomorphic to an $ M $-subcomodule of a canonical $ M $-comodule, which may be taken to be the canonical W*-$ M $-comodule $ (B(H)\vt M,\mathrm{id }_{B(H)}\otimes\Delta) $ for some Hilbert space $ H $.
	
	Indeed, the image $ \alpha(X) $ of $ X $ under the action $ \alpha $ is an $ M $-subcomodule of the canonical $ M $-comodule $ X\ft M $, since we have:
	\[(\mathrm{id }_{X}\otimes\Delta)\circ\alpha(X)=(\alpha\otimes \mathrm{id }_{M})\circ\alpha(X)\sub\alpha(X)\ft M \] 
	and $ \alpha $ is an $ M $-comodule isomorphism of $ X $ onto $ \alpha(X) $ and thus
	\[X\simeq\alpha(X)\leq X\ft M. \]
	Furthermore, we may suppose that $ X\sub B(H) $ as a w*-closed subspace for some Hilbert space $ H $, thus $ X\ft M \leq B(H)\vt M $.
\end{remark}

\begin{remark}\label{rem3}
	For any Hopf-von Neumann algebra $ (M,\Delta) $, the predual $ M_{*} $ of $ M $ becomes naturally a Banach algebra with the product defined by
	\[\omega \varphi=(\omega\otimes \varphi)\circ\Delta, \]
	for $ \omega, \varphi\in M_{*} $. Furthermore, an $ M $-comodule $ (X,\alpha) $ becomes an $ M_{*} $-Banach module with the module operation defined as
	\[\omega\cdot x=(\mathrm{id }_{X}\otimes \omega)\circ\alpha(x),\quad \omega\in  M_{*},\ x\in X. \] It is easy to verify that the $ M $-subcomodules of $ X $ are exactly the $ M_{*} $-submodules of $ X $ with respect to the above $ M_{*} $-module action.
	
	Also, it is easy to see that a w*-continuous complete contraction $ \phi\colon X\to Y $ between two $ M $-comodules $ X $ and $ Y $ is an $ M $-comodule morphism if and only if $ \phi $ is an $ M_{*} $-module homomorphism.
\end{remark}

The notion of fixed points  is of great importance in the study of comodules of Hopf-von Neumann algebras and in the study of crossed products as will be obvious in the next sections. 

\begin{defin} 
Let $ (X,\alpha) $ be an $ M $-comodule over a Hopf-von Neumann algebra $ (M,\Delta) $. The \bfc{fixed point subspace} of $ X $ is the operator space
\[X^{\alpha}=\{x\in X:\ \alpha(x)=x\otimes1_{M} \}.\] 	
\end{defin}

Note that $ X^{\alpha} $ is obviously an $ M $-subcomodule of $X$. Furthermore, it is easy to see that any $ M $-comodule isomorphism $ \phi\colon X\to Y $ between two $ M $-comodules $ (X,\alpha) $ and $ (Y,\beta) $ maps $ X^{\alpha} $ onto $ Y^{\beta} $.

Another important notion concerning actions of Hopf-von Neumann algebras is \bfc{commutativity of actions}:

\begin{defin}\label{def2.3} 
	Let $ (M_{1},\Delta_{1}) $ and $ (M_{2},\Delta_{2}) $ be two Hopf-von Neumann algebras and $ \alpha_{1} $,  $ \alpha_{2} $ be actions of $ M_{1} $ and $ M_{2} $ on the same operator space $ X $ respectively. We say that $ \alpha_{1} $ and  $ \alpha_{2} $ \bfc{commute} if
	\[(\alpha_{1}\otimes \mathrm{id }_{M_{2}})\circ\alpha_{2}=(\mathrm{id }_{X}\otimes\sigma)\circ(\alpha_{2}\otimes \mathrm{id }_{M_{1}})\circ\alpha_{1}, \]
	where $ \sigma\colon M_{2}\vt M_{1}\to M_{1}\vt M_{2}:\ x\otimes y\mapsto y\otimes x $ is the flip isomorphism.
\end{defin}

The next lemma due to Hamana \cite{Ha1} essentially states that commuting actions can yield new (non-trivial) actions on the fixed point space of a given comodule and it will be very useful in the following.
 
\begin{lem}\label{lem2.3}\cite[Lemma 5.2]{Ha1} 
	If $ \alpha_{1} $ and $ \alpha_{2} $ are commuting actions on the same operator space $ X $ (Definition \ref{def2.3}) of two Hopf-von Neumann algebras $ M_{1} $ and $ M_{2} $ respectively, then the fixed point subspace $X^{\alpha_{1}} $ is an $ M_{2} $-subcomodule of $ (X,\alpha_{2}) $, i.e. the restriction $ \alpha_{2}|_{X^{\alpha_{1}}} $ is an action of $ M_{2} $ on $ X^{\alpha_{1}} $.
\end{lem}

\subsection{Hopf-von Neumann algebras associated with groups}

For the rest of this paper, $ G $ will denote a locally compact (Hausdorff) group with left Haar measure $ ds $ and modular function $ \Delta_{G} $.

As usual, we identify $L^{\infty}(G)  $ with the von Neumann subalgebra of $ \BLT $ of multiplicative operators acting on $ L^{2}(G) $. Identifying $ L^{\infty}(G\times G)$ with $ L^{\infty}(G)\vt L^{\infty}(G) $ it is not hard to check that the map $ \alpha_{G}\colon L^{\infty}(G)\to L^{\infty}(G)\vt L^{\infty}(G) $ given by
\[ \alpha_{G}(f)(s,t)=f(ts),\quad s, t\in G, \] is a comultiplication on $ \LI $.

 Also, identifying $  B(L^{2}(G))\vt B(L^{2}(G))$ with $ B(L^{2}(G\times G)) $ one can see that the fundamental unitary operator $ V_{G}\in B(L^{2}(G\times G))$  defined by the formula 
\[V_{G}f(s,t)=f(t^{-1}s,t),\quad f\in L^{2}(G\times G),\ s, t\in G, \]
induces $ \alpha_{G} $, that is: 
\[\alpha_{G}(f)=V_{G}^{*}(f\otimes1)V_{G},\quad f\in \LI. \] 

 So $ (L^{\infty}(G),\alpha_{G}) $ is a Hopf-von Neumann algebra. In addition, the product defined on the predual $ \LO\simeq \LI_{*} $ by $ \alpha_{G} $ (see Remark \ref{rem3}) is given by 
 \[hk=k\ast h,\quad\forall h,k\in\LO, \]
 where \[(k\ast h)(t)=\int_{G}k(s)h(s^{-1}t)\ ds,\quad t\in G,\] is the usual convolution on $ \LO $.

Another basic example of a Hopf-von Neumann algebra associated with the group $ G $ is the left von Neumann algebra of $ G $, i.e. the algebra $ L(G):= \lambda(G)''\subseteq B(L^{2}(G)) $ generated by the left regular representation $ \lambda\colon G\ni s\mapsto\lambda_{s}\in B(L^{2}(G)) $,
\[\lambda_{s}\xi(t)=\xi(s^{-1}t) , \qquad\xi\in L^{2}(G). \]

The comultiplication on $ L(G) $ is given by the map $ \delta_{G}\colon L(G)\to L(G)\vt L(G) $ with
\[\delta_{G}(\lambda_{s})=\lambda_{s}\otimes\lambda_{s},\quad s\in G \]
and it is easily verified that  $ \delta_{G} $ is induced by the fundamental unitary operator $ W_{G}\in B(L^{2}(G\times G)) $, given by the formula
\[W_{G}f(s,t)=f(s,st),\quad f\in L^{2}(G\times G),\ s,t\in G. \]
This means that
\[\delta_{G}(x)=W_{G}^{*}(x\otimes1)W_{G},\quad x\in L(G). \]

The predual $ L(G)_{*} $ of $ L(G) $ is isometrically isomorphic to the Fourier algebra $ A(G) $ of $ G $ (see e.g. \cite{Ey}):
\[A(G)=\{u\colon G\to \mathbb{C}:\ \exists\xi, \eta\in\LT,\ \forall s\in G,\ u(s)=\la \lambda_{s}\xi,\ \eta \ra \}. \]
The pointwise product on $ A(G) $  coincides with that induced on the predual $ L(G)_{*} $ by the comultiplication $ \delta_{G} $ of $ L(G) $ (see Remark \ref{rem3}), because:
\[\la \lambda_{s} ,\ uv\ra=u(s)v(s)=\la \lambda_{s} ,\ u\ra\la \lambda_{s} ,\ v\ra=\la \lambda_{s}\ot\lambda_{s} ,\ u\ot v\ra=\la \delta_{G}(\lambda_{s}),\ u\ot v\ra. \]

Note that $ V_{G}\in L(G)\vt L^{\infty}(G) $ and $ W_{G}\in L^{\infty}(G)\vt L(G) $ and thus $ \alpha_{G} $ and $ \delta_{G} $ extend to actions of $ L^{\infty}(G) $ and $ L(G) $ on $ B(L^{2}(G)) $ respectively, via the formulas
\[ \alpha_{G}(x)=V_{G}^{*}(x\otimes1)V_{G},\quad x\in B(L^{2}(G)),  \]
\[\delta_{G}(x)=W_{G}^{*}(x\otimes1)W_{G}, \quad x\in B(L^{2}(G)). \]
Also, we will need the $ \LI $-action $ \beta_{G}\colon\BLT\to\BLT\vt\LI $, with \[\beta_{G}(x)=U_{G}^{*}(x\ot1)U_{G},\quad x\in\BLT, \]
where $ U_{G} $ is the unitary on $ L^{2}(G\times G) $ defined by \[U_{G}f(s,t)=\Delta_{G}(t)^{1/2}f(st,t),\quad f\in L^{2}(G\times G),\ s,t\in G.\]

 It is not hard to check that $ U_{G}\in R(G)\vt\LI $, where $ R(G)=\rho(G)''=L(G)' $ is the right group von Neumann algebra generated by the right  regular representation of $ G $ on $ L^{2}(G) $:
\[\rho_{s}f(t)=\Delta_{G}(s)^{1/2}f(ts),\qquad s, t\in G,\ f\in L^{2}(G). \]

Furthermore, if $ \sigma\colon\BLT\vt\BLT\to\BLT\ot\BLT:\ x\ot y\mapsto y\ot x $ is the flip isomorphism, then one can verify that 
\[\sigma\circ\alpha_{G}(f)=U_{G}(f\ot1)U_{G}^{*}, \text{ for all }  f\in \LI\]

Since $ \LI'=\LI $ and $ L(G)'=R(G) $ one can easily verify the following 
\[\BLT^{\delta_{G}}=\LI \] and \[ \BLT^{\beta_{G}}=L(G). \] 
	
\section{Non-degeneracy and saturation of comodules}\label{sec3}

Here we deal with the notions of non-degeneracy and saturation for general comodules. The results of this section may seem too abstract, but they will be crucial in the following nevertheless.

The importance of the notions of non-degeneracy and saturation lies in that they are necessary in order to characterize those comodules that satisfy the Takesaki-duality (see section \ref{sec5}). The term saturation was introduced in \cite{SVZ} for W*-comodules, while the term non-degeneracy was systematically used in \cite{VH} also for W*-comodules.

In \cite{SVZ1} Str\v{a}til\v{a}, Voiculescu and Zsid\'{o} proved that every W*-$ \LI $-comodule is non-degenerate and consequently that every W*-$ \LI $-comodule is saturated. Using the saturation property they extended the Takesaki-duality for W*-$ \LI $-comodules to the case of non-abelian groups.

It was proved in \cite{SVZ} that every W*-$ L(G) $-comodule is saturated if $ G $ is amenable and the amenability assumption was later removed by Landstad (see \cite[Proposition II.1.1]{SVZ2} and the remark after that). Using this, Str\v{a}til\v{a}, Voiculescu and Zsid\'{o} proved that the Takesaki-duality holds for W*-$ L(G) $-comodules (\cite[Theorem II.2.1]{SVZ2}). This time, they followed the dual path using the saturation property to show that every W*-$ L(G) $-comodule is non-degenerate and thus it satisfies the Takesaki-duality.

The same results were independently obtained by Landstad \cite{La2} and Nakagami \cite{Nak}. An alternative proof of the Takesaki-duality was given by Van Heeswijck \cite{VH} who proved the non-degeneracy for W*-$ \LI $-comodules and W*-$ L(G) $-comodules without using the saturation property.

As will be proved in section \ref{sec5}, the saturation property for $ \LI $-comodules and $ L(G) $-comodules (which are not von Neumann algebras) is equivalent to the Takesaki-duality for the respective Fubini crossed products whereas non-degeneracy is equivalent to the Takesaki-duality for the spatial crossed products. Also, we will see that every $ \LI $-comodule is non-degenerate and saturated, but this is not the case for the category of $ L(G) $-comodules, unless $ G $ has the AP (see section \ref{sec4}).

\begin{defin}\label{d1} 
	Let $ (M,\Delta) $ be a Hopf-von Neumann algebra acting on a Hilbert space $ K $ and  $ (X,\alpha) $ be an $ M $-comodule with $ X $ being a w*-closed subspace of $ B(H) $ for a Hilbert space $ H $. Then, $ (X,\alpha) $ is called \bfc{non-degenerate} if 
	\[X\vt B(K)=\wsp\{(1_{H}\ot b)\alpha(x):\ x\in X,\ b\in B(K) \}. \]
\end{defin}

\begin{remark}\label{rem8} Suppose that $ (M,\Delta) $ and $ (X,\alpha) $ are as in Definition \ref{d1} and let $ (Y,\beta) $ be an $ M $-comodule with $ Y $ being a w*-closed subspace of $ B(L) $ for some Hilbert space $ L $. If $ \phi\colon X\to Y $ is an $ M $-comodule isomorphism and $ X $ is non-degenerate, then $ Y $ is non-degenerate too.
	
Indeed, since $ \phi\colon X\to Y $ is a w*-bicontinuous completely isometric isomorphism, so is the map $ \psi:=\phi\ot\Id\colon X\vt B(K)\to Y\vt B(K) $ and clearly $ \psi $ satisfies the following:
\[\psi((1_{H}\ot b)z)=(1_{L}\ot b)\psi(z), \quad \text{ for any }z\in X\vt B(K) \text{ and }b \in B(K). \]	  
Also, since $ \beta\circ \phi=(\phi\ot \Id)\circ\alpha $ and $ \phi(X)=Y $ it follows that $ \psi(\alpha(X))=\beta(Y) $. Thus if $ X $ is non-degenerate, then so is $ Y $.

In particular, the non-degeneracy of $ (X,\alpha) $ does not depend on the Hilbert space $ H $ on which $ X $ is represented.
\end{remark} 

\begin{pro}\label{p1} If $ (M,\Delta) $ is a Hopf-von Neumann algebra acting on a Hilbert space $ K $ and $ (X,\alpha) $ is a non-degenerate $ M $-comodule with $ X\sub B(H) $ (w*-closed), then $ X=\wsp\{M_{*}\cdot X\} $.	
\end{pro}
\begin{proof} Let $ \phi\in X_{*} $, such that $ \phi(\omega\cdot x)=0 $, for all $ \omega\in M_{*}$ and $ x\in X $. Then, we have:
	\begin{align*} 
	&\phi\circ(\mathrm{id}_{X}\ot\omega)\circ\alpha(x)=0,\quad \forall\omega\in M_{*},\ \forall x\in X\\
	\implies&\omega\circ(\phi\ot \mathrm{id}_{B(K)})\circ\alpha(x)=0,\quad \forall\omega\in M_{*},\ \forall x\in X\\
	\implies&(\phi\ot \mathrm{id}_{B(K)})\circ\alpha(x)=0,\quad \forall x\in X\\
	\implies& b(\phi\ot \mathrm{id}_{B(K)})\circ\alpha(x)=0,\quad \forall b\in B(K),\ \forall x\in X\\
	\implies&(\phi\ot \mathrm{id}_{B(K)})\left((1_{H}\ot b) \alpha(x)\right) =0,\quad \forall b\in B(K),\ \forall x\in X.\\
	\end{align*}
	Since $ (X,\alpha) $ is non-degenerate, the last condition implies that $ (\phi\ot \mathrm{id}_{B(K)})(y)=0 $ for any $ y\in X\vt B(K) $, thus $ \phi(x)1=(\phi\ot \mathrm{id}_{B(K)})(x\ot1)=0 $	for any $ x\in X $ and hence $ \phi=0 $. So the desired conclusion follows from the Hahn-Banach theorem.	
\end{proof}

\begin{defin}\label{d2} Let $ (M,\Delta) $ be a Hopf-von Neumann algebra and $ (X,\alpha) $ be an $ M $-comodule. The \bfc{saturation space} of $ (X,\alpha) $ is the space
	\[\mathrm{Sat}(X,\alpha):=\{y\in X\ft M:\ (\mathrm{id}_{X}\ot\Delta)(y)=(\alpha\ot \mathrm{id}_{M})(y) \}. \]
	Obviously, $ \alpha(X)\sub \mathrm{Sat}(X,\alpha) $. We say that $ (X,\alpha) $ is \bfc{saturated} if $ \alpha(X)= \mathrm{Sat}(X,\alpha) $. 
\end{defin}

\begin{pro}\label{p2} Let $ (M,\Delta) $ and $ (X,\alpha) $ be as in Definition \ref{d2}. Then the following hold:	
	\begin{itemize}
		\item[(i)] The saturation space $ \mathrm{Sat}(X,\alpha) $ is an $ M $-subcomodule of the canonical $ M $-comodule $ (X\ft M, \mathrm{id}_{X}\ot\Delta) $;
		\item[(ii)] For the $ M_{*} $-module action on $ \mathrm{Sat}(X,\alpha) $ defined by the canonical $ M $-action $ \Id_{X}\ot\Delta $, we have $ M_{*}\cdot \mathrm{Sat}(X,\alpha) \sub \alpha(X) $;
		\item[(iii)] The $ M $-comodule  $ (\mathrm{Sat}(X,\alpha),\mathrm{id}_{X}\ot\Delta) $ is non-degenerate if and only if $ (X,\alpha) $ is non-degenerate and saturated.
	\end{itemize}	 
\end{pro}
\begin{proof} (i) Let $ y\in \mathrm{Sat}(X,\alpha) $. Then, $ (\mathrm{id}_{X}\ot\Delta)(y)=(\alpha\ot \mathrm{id}_{M})(y)\in\alpha(X)\ft M \sub \mathrm{Sat}(X,\alpha)\ft M$. Thus, $ \mathrm{Sat}(X,\alpha) $ is an $ M $-subcomodule of  $ (X\ft M, \mathrm{id}_{X}\ot\Delta) $.
	
	(ii) Let $ \omega\in M_{*} $ and $ y\in \mathrm{Sat}(X,\alpha)$. Since $ \mathrm{Sat}(X,\alpha)\sub X\ft M $, we have that $ (\mathrm{id}_{X}\ot\omega)(y)\in X $. Therefore, we get:
	\begin{align*} 
	\omega\cdot y&=(\mathrm{id}_{X}\ot \mathrm{id}_{M}\ot\omega)\circ(\mathrm{id}_{X}\ot\Delta)(y)\\
	&=(\mathrm{id}_{X}\ot \mathrm{id}_{M}\ot\omega)\circ(\alpha\ot \mathrm{id}_{M})(y)\\
	&=\alpha\circ(\mathrm{id}_{X}\ot\omega)(y)\in \alpha(X),
	\end{align*}	
	thus $ M_{*}\cdot \mathrm{Sat}(X,\alpha) \sub \alpha(X) $.
	
	(iii) Suppose that 	$ (\mathrm{Sat}(X,\alpha),\mathrm{id}_{X}\ot\Delta) $ is non-degenerate. Then, by Proposition \ref{p1} and Proposition \ref{p2} (ii), it follows immediately that 
	\[\mathrm{Sat}(X,\alpha)=\wsp\{M_{*}\cdot \mathrm{Sat}(X,\alpha)\}\sub\alpha(X),\]
	therefore $ \mathrm{Sat}(X,\alpha)=\alpha(X) $, i.e. $ (X,\alpha) $ is saturated. On the other hand, $ (X,\alpha) $ is isomorphic with $ (\alpha(X),\mathrm{id}_{X}\ot\Delta) $, which is non-degenerate since $ \mathrm{Sat}(X,\alpha)=\alpha(X) $. Thus, $ (X,\alpha) $ is non-degenerate.
	
	Conversely, suppose that $ (X,\alpha) $ is non-degenerate and saturated. Then, since $ (X,\alpha)\simeq(\alpha(X),\mathrm{id}_{X}\ot\Delta)=(\mathrm{Sat}(X,\alpha),\mathrm{id}_{X}\ot\Delta) $, it follows that $ (\mathrm{Sat}(X,\alpha),\mathrm{id}_{X}\ot\Delta) $ is non-degenerate.	
\end{proof}

\begin{remark}\label{rem7iii} (iii) Let $ (M,\Delta) $ be a Hopf-von Neumann algebra and let $ (Y_{i},\delta_{i}) $ for $ i=1,2 $ be two $M $-comodules. Also, let $ \phi\colon Y_{1}\to Y_{2} $ be an $ M $-comodule isomorphism. Then, the map $ \phi\ot\Id_{M}\colon Y_{1}\ft M\to Y_{2}\ft M $ is an $ M $-comodule isomorphism for the canonical actions $ \Id_{Y_{i}}\ot\Delta $, $ i=1,2 $. Furthermore, $ \phi\ot\Id_{M} $ maps $ \mathrm{Sat}(Y_{1},\delta_{1}) $ onto $ \mathrm{Sat}(Y_{2},\delta_{2}) $. Indeed, for any $ x\in Y_{1}\ft M $, we have:
	\begin{align*} 
	&(\phi\ot\Id_{M})(x)\in \mathrm{Sat}(Y_{2},\delta_{2})\\\iff &(\delta_{2}\ot \Id_{M})\circ(\phi\ot\Id_{M})(x)=(\Id_{Y_{2}}\ot\Delta)\circ(\phi\ot\Id_{M})(x)\\
	\iff&((\delta_{2}\circ\phi)\ot\Id_{M})(x)=(\phi\ot\Id_{M}\ot\Id_{M})\circ(\Id_{Y_{1}}\ot\Delta)(x)\\
	\iff&\left[ ((\phi\ot\Id_{M})\circ\delta_{1})\ot\Id_{M}\right] (x)=(\phi\ot\Id_{M}\ot\Id_{M})\circ(\Id_{Y_{1}}\ot\Delta)(x)\\
	\iff&(\phi\ot\Id_{M}\ot\Id_{M})\circ (\delta_{1}\ot\Id_{M})(x)=\\&(\phi\ot\Id_{M}\ot\Id_{M})\circ(\Id_{Y_{1}}\ot\Delta)(x) \\
	\iff&(\delta_{1}\ot\Id_{M})(x)=(\Id_{Y_{1}}\ot\Delta)(x)\\
	\iff& x\in \mathrm{Sat}(Y_{1},\delta_{1}).
	\end{align*}
	Since $ \phi\ot\Id_{M} $ is onto $ Y_{2}\ft M $, the above equivalences show that it maps $ \mathrm{Sat}(Y_{1},\delta_{1}) $ onto $ \mathrm{Sat}(Y_{2},\delta_{2}) $. Therefore, saturation is preserved by comodule isomorphisms.
	
\end{remark}

The following corollary follows immediately from  Proposition \ref{p2} (iii).

\begin{cor}\label{c1}Let $ (M,\Delta) $ be a Hopf-von Neumann algebra. If every $ M $-comodule is non-degenerate, then every $ M $-comodule is saturated.	
\end{cor}

\begin{remark} 
	Note that it does not necessarily follow from Corollary \ref{c1} or its proof that every non-degenerate $ M $-comodule is saturated.
\end{remark}
 
\begin{pro}\label{p8} For a Hopf-von Neumann algebra $ (M,\Delta) $ the following conditions are equivalent:
	\begin{itemize}	
		\item[(a)] Every $ M $-comodule is saturated;
		\item[(b)] For any  $ M $-comodule $ (X,\alpha) $, any $ M $-subcomodule $ Z $ of $ X $  and any $ x\in X $, the following implication holds: \[ \alpha(x)\in Z\ft M \implies x\in Z ;\]
		\item [(c)] For any  $ M $-comodule $ (X,\alpha) $ and any $ x\in X $, we have $ x\in \overline{M_{*}\cdot x}^{\mathrm{w}^{*}} $.
	\end{itemize}	
\end{pro}
\begin{proof} (a)$ \implies $(b): Suppose that every $ M $-comodule is saturated. Let $ (X,\alpha) $ be an $ M $-comodule, $ Z $ an $ M $-subcomodule of $ X $ and $ x\in X $ such that $ \alpha(x)\in Z\ft M $. By assumption, $ \alpha $ restricts to an $ M $-action on $ Z $ and since $ \alpha(x)\in Z\ft M $ and $ (\alpha\ot \mathrm{id})(\alpha(x))=(\mathrm{id}\ot \Delta)(\alpha(x)) $ it follows that $ \alpha(x)\in \mathrm{Sat}(Z,\alpha|_{Z}) $. But $ (Z,\alpha|_{Z}) $ is saturated by hypothesis and therefore $ \alpha(x)\in \alpha(Z) $. Thus, $ x\in Z $, because $ \alpha $ is an isometry.
	
	(b)$ \implies $(c): Let $ (X,\alpha) $ be an $ M $-comodule  and $ x\in X $ and put $ Z:=\overline{M_{*}\cdot x}^{\text{w*}}  $. Then, by Remark \ref{rem3}, it follows that $ Z $ is an $ M $-subcomodule of $ X $ since it is an $ M_{*} $-module by definition.
	
	Also, we have that $ \alpha(x)\in Z\ft M $. Indeed, by the definition of the Fubini tensor product, the condition $ \alpha(x)\in Z\ft M $ is equivalent to the following \[\omega\cdot x=(\mathrm{id}\ot \omega)(\alpha(x))\in Z,\quad\forall\omega\in M_{*}, \]
	which is true by the definition of $ Z $. Therefore, the assumption that (b) holds implies that $ x\in Z $, that is $ x\in \overline{M_{*}\cdot x}^{\text{w*}} $.
	
	(c)$ \implies $(a): Suppose that (c) is true and take an $ M $-comodule $ (Y,\beta) $. Consider the $ M $-comodule $ (X,\alpha) $ with $ X:=\mathrm{Sat}(Y,\beta) $ and $ \alpha=\mathrm{id}_{Y}\ot \Delta $. Then, by (c), it follows that $ z\in \overline{M_{*}\cdot z}^{\text{w*}} \sub \overline{M_{*}\cdot \mathrm{Sat}(Y,\beta)}^{\text{w*}} $, for all $ z\in \mathrm{Sat}(Y,\beta) $. But from Proposition \ref{p2} (ii), we have that $ M_{*}\cdot \mathrm{Sat}(Y,\beta)\sub \beta(Y) $ and therefore $ z\in \beta(Y) $, for all $ z\in \mathrm{Sat}(Y,\beta) $, that is $ (Y,\beta) $ is saturated.	
\end{proof}

The next two lemmas describe two basic ways of constructing new saturated comodules. 

\begin{lem}\label{l1} Let $ (M,\Delta) $ be a Hopf-von Neumann algebra and $ (Y,\beta) $ be a saturated $ M $-comodule. Then, $ (X\ft Y, \Id_{X}\ot\beta) $ is a saturated $ M $-comodule for any dual operator space $ X $.	 
\end{lem}
\begin{proof} Let $ X $ be a dual operator space. First we have to check that $ (X\ft Y, \Id_{X}\ot\beta) $ is an $ M $-comodule. Indeed, we have:
	\begin{align*} 
	    (\Id_{X}\ot\beta\ot\Id_{M})\circ(\Id_{X}\ot\beta)&=\Id_{X}\ot\left[ (\beta\ot \Id_{M})\circ\beta\right]\\
	    &=\Id_{X}\ot\left[ (\Id_{Y}\ot\Delta)\circ\beta\right]\\
	    &=(\Id_{X}\ot\Id_{Y}\ot\Delta)\circ(\Id_{X}\ot\beta).	    
	\end{align*}
	
Take $ z\in \mathrm{Sat}(X\ft Y,\Id_{X}\ot\beta) $. We have to prove that $ z\in(\Id_{X}\ot\beta)(X\ft Y)=X\ft\beta(Y) $. Indeed, since $ z\in \mathrm{Sat}(X\ft Y,\Id_{X}\ot\beta) $ we have:
	\[(\Id_{X}\ot\Id_{Y}\ot\Delta)(z)=(\Id_{X}\ot\beta\ot\Id_{M})(z).\]
	Therefore, for any $ \omega\in X_{*} $, we get:
	\begin{equation*} 
	(\omega\ot\Id_{Y\ft M\ft M})\circ(\Id_{X}\ot\Id_{Y}\ot\Delta)(z)=(\omega\ot\Id_{Y\ft M\ft M})\circ(\Id_{X}\ot\beta\ot\Id_{M})(z)
	\end{equation*}
	that is 
	\begin{equation*} 
	(\Id_{Y}\ot\Delta)\circ(\omega\ot\Id_{Y\ft M})(z)=(\beta\ot\Id_{M})\circ(\omega\ot\Id_{Y\ft M})(z).
	\end{equation*}
	Thus $ (\omega\ot\Id_{Y\ft M})(z)\in \mathrm{Sat}(Y,\beta)=\beta(Y) $ for all 	$ \omega\in X_{*} $ and hence $ z\in X\ft\beta(Y) $.	
\end{proof}

\begin{lem}\label{l2} Let $ M_{1} $ and $ M_{2} $ be two Hopf-von Neumann algebras and let  $ \alpha_{1} $ and $ \alpha_{2} $ be actions of $ M_{1} $ and $ M_{2} $ respectively on the same dual operator space $ X $. Suppose that $ (X,\alpha_{2}) $ is a saturated $ M_{2} $-comodule and that $ \alpha_{1} $ and $ \alpha_{2} $ commute, i.e. \[(\alpha_{1}\otimes \mathrm{id }_{M_{2}})\circ\alpha_{2}=(\mathrm{id }_{X}\otimes\sigma)\circ(\alpha_{2}\otimes \mathrm{id }_{M_{1}})\circ\alpha_{1}, \]
	where $ \sigma\colon M_{2}\vt M_{1}\to M_{1}\vt M_{2}:\ x\otimes y\mapsto y\otimes x $ is the flip isomorphism. Then, the fixed point space $ (X^{\alpha_{1}},\alpha_{2}|_{X^{\alpha_{1}}}) $ is a saturated $ M_{2} $-comodule.	
\end{lem}
\begin{proof} Since the actions $ \alpha_{1} $ and $ \alpha_{2} $ commute, $ X^{\alpha_{1}} $ is an $ M_{2} $-subcomodule of $ (X,\alpha_{2}) $ by Lemma \ref{lem2.3}. Also, since $ (X,\alpha_{2}) $ is saturated we have $ \mathrm{Sat}(X,\alpha_{2})=\alpha_{2}(X) $ and therefore
	\begin{align*} 
	\mathrm{Sat}(X^{\alpha_{1}},\alpha_{2}|_{X^{\alpha_{1}}})&=(X^{\alpha_{1}}\ft M)\cap \mathrm{Sat}(X,\alpha_{2})\\
	&=(X^{\alpha_{1}}\ft M)\cap \alpha_{2}(X).
	\end{align*}	
	Thus it suffices to show that $ (X^{\alpha_{1}}\ft M_{2})\cap \alpha_{2}(X)\sub\alpha_{2}(X^{\alpha_{1}}) $.
	
	 Take $ y\in (X^{\alpha_{1}}\ft M_{2})\cap \alpha_{2}(X) $. Then $ y=\alpha_{2}(x) $ for some $ x\in X$ and so we only need to prove that $ x\in X^{\alpha_{1}} $, i.e. $ \alpha_{1}(x)=x\ot1 $.	
	 Indeed, since $ y\in X^{\alpha_{1}}\ft M_{2} $ it follows that
	\[(\Id_{X}\ot\sigma)\circ(\alpha_{1}\ot\Id_{M_{2}})(y)=y\ot 1 \]
	and therefore
	\begin{align*} 
	(\alpha_{2}\ot\Id_{M_{1}})(x\ot1)&=\alpha_{2}(x)\ot1\\
	&=y\ot1\\&=(\Id_{X}\ot\sigma)\circ(\alpha_{1}\ot\Id_{M_{2}})(y)\\&=(\Id_{X}\ot\sigma)\circ(\alpha_{1}\ot\Id_{M_{2}})(\alpha_{2}(x))\\&=(\alpha_{2}\ot\Id_{M_{1}})(\alpha_{1}(x)),
	\end{align*}
	where the last equality follows from the commutativity of the actions $ \alpha_{1} $ and $ \alpha_{2} $. Since $ \alpha_{2}\ot\Id_{M_{1}} $ is an isometry it follows that $ \alpha_{1}(x)=x\ot 1 $ and the proof is complete.	
\end{proof}

\section{Crossed products and the approximation property}\label{sec4}

In this section we consider the (dual) Hopf-von Neumann algebras $ (\LI,\alpha_{G}) $ and $ (L(G),\delta_{G}) $ and we study the basic properties of their comodules and the associated crossed products.

\subsection{Crossed products of $ \LI $-comodules}

Perhaps the most important property of $ (\LI,\alpha_{G}) $ is that every $ \LI $-comodule is non-degenerate and saturated (see Lemma \ref{c2} below). This is the key ingredient in the proofs of some of the main results in the following (see e.g. Proposition \ref{p10}, Theorem \ref{thm1} and Propositions \ref{p14} and \ref{p16}).

Before we proceed to the study of crossed products of $ \LI $-comodules, let us recall some known results from the theory of crossed products of von Neumann algebras.

Let $ M $ be a von Neumann algebra and let $ \gamma\colon G\to\mathrm{Aut}(M) $ be a (pointwise) action of $ G $ on $ M $, i.e. a group homomorphism from $ G $ to the group of unital normal *-automorphisms of $ M $ such that the function $$ G\ni s\mapsto\gamma_{s}(x)\in M $$ is w*-continuous for any $ x\in M $. Then, we get a W*-$ \LI $-action $ \alpha\colon M\to M\vt\LI $ given by  \[\la\alpha(x),\omega\ot f\ra=\int_{G}\la\gamma_{s^{-1}}(x),\omega\ra f(s)\ ds,\quad x\in M,\ \omega\in M_{*}, f\in \LO. \]
The map $ \alpha $ is clearly a w*-continuous unital *-monomorphism (and thus completely isometric) and the fact that $ \gamma_{s}\circ\gamma_{t} $ for all $ s,t\in G $ yields that 
\[(\alpha\ot\Id)\circ\alpha=(\Id\ot\alpha_{G})\circ\alpha \]
and therefore $ \alpha $ is indeed a W*-$ \LI $-action. In fact, every W*-$ \LI $-action on $ M $ arises this way (see for example \cite[18.6]{Stra}). Furthermore, it is easy to verify that the fixed points of the action $ \gamma $ are exactly the fixed point subspace $ M^{\alpha} $, that is, for an $ x\in M $, we have that $ \alpha(x)=x\ot 1 $ if and only if $ \gamma_{s}(x)=x $ for all $ s\in G $.

The crossed product $ M\rtimes_{\alpha}G $ (or $ M\rtimes_{\gamma}G $) is defined as the von Neumann subalgebra of $ M\vt\BLT $ generated by $ \alpha(M) $ and $ \mathbb{C}1\vt L(G) $.

According to the Digernes-Takesaki theorem (see for example \cite[Chapter X, Corollary 1.22]{Tak}) we have that $ M\rtimes_{\alpha}G $ is equal to the fixed point algebra of the action $ \beta $ of $ G $ on $ M\vt\BLT $ defined by
\[\beta_{s}=\gamma_{s}\ot\Ad\rho_{s},\quad s\in G. \]

On the other hand, if $ \alpha $ is the W*-$ \LI $-action associated to the action $ \gamma $, then one can verify (see also \cite{NaTa} page 9) that the W*-$ \LI $-action  corresponding to the pointwise action $ \beta $, which we denote by $ \wt{\alpha} $, is given directly by $ \alpha $ via the formula:
\[\widetilde{\alpha}=(\mathrm{id }_{M}\otimes \mathrm{Ad }U_{G}^{*})\circ(\mathrm{id }_{M}\otimes\sigma)\circ(\alpha\otimes \mathrm{id }_{B(L^{2}(G))}) \] 	where $ \sigma $ is the flip isomorphism on $\BLT\vt\BLT$. Therefore, the Digernes-Takesaki theorem is rephrased as \[M\rtimes_{\alpha}G=\left( M\vt\BLT\right)^{\wt{\alpha}}.  \]

Taking into consideration all of the above, Hamana \cite{Ha1} suggested the Definitions \ref{deffcr} and \ref{def4.2} below. 

\begin{defin}\label{deffcr}
	 For an $ \LI $-comodule $ (X,\alpha) $, we define the map \[ \widetilde{\alpha} \colon X \vt B(L^{2}(G)) \to X \vt B(L^{2}(G)) \vt L^{\ap}(G) \] by
		\[\widetilde{\alpha}=(\mathrm{id }_{X}\otimes \mathrm{Ad }U_{G}^{*})\circ(\mathrm{id }_{X}\otimes\sigma)\circ(\alpha\otimes \mathrm{id }_{B(L^{2}(G))}), \]
		where $ \sigma $ is the flip isomorphism on $\BLT\vt\BLT$.	
\end{defin}

The next result is essentially the same as \cite[Lemma 5.3 (i)]{Ha1} with the appropriate modifications since Hamana considers on $ \LI $ the (opposite) comultiplication $ \sigma\circ\alpha_{G} $ and uses the right group von Neumann algebra $ R(G) $ instead of $ L(G) $ as the dual of $ \LI $. Therefore we omit its proof. 

\begin{pro} \label{pro4.2i} Let $ (X,\alpha) $ be an $ \LI $-comodule. Then, $ \widetilde{\alpha} $ is an $ \LI $-action on $ X \vt B(L^{2}(G)) $, which commutes with the $ L(G) $-action $ \mathrm{id }_{X}\otimes\delta_{G} $.	
\end{pro}

\begin{defin}\label{def4.2}
	Let $ (X,\alpha) $ be an $ \LI $-comodule. The \bfc{Fubini crossed product} of $ X $ by $ \alpha $ is defined to be the $ L(G) $-comodule $ (X\rtimes^{\cl{F}}_{\alpha}G,\ \widehat{\alpha}) $, where
	\[X\rtimes^{\cl{F}}_{\alpha}G:=(X\vt B(L^{2}(G)))^{\widetilde{\alpha}} \]
	and \[\widehat{\alpha}:=(\mathrm{id }_{X}\otimes\delta_{G})|_{X\fcr_{\alpha}G}. \]
	The $ L(G) $-action $ \widehat{\alpha}\colon X\fcr_{\alpha}G\to (X\fcr_{\alpha}G)\ft L(G) $ is called the \bfc{dual action} of $ \alpha $. 
\end{defin}

By Proposition \ref{pro4.2i} and Lemma \ref{lem2.3} we get that $ (X\fcr_{\alpha}G,\wh{\alpha}) $ is indeed an $ L(G) $-subcomodule of $ (X\vt\BLT,\Id_{X}\ot\delta_{G}) $.

\begin{defin}\label{def4.4}
	Let $ (X,\alpha) $ be an $ \LI $-comodule and suppose that $ X $ is a w*-closed subspace of $ B(H) $ for some Hilbert space $ H $. The \bfc{spatial crossed product} of $ X $ by $ \alpha $ is defined to be the space
	\begin{align*}
	X\scr_{\alpha}G&:=\wsp\{(1_{H}\ot\lambda_{s})\alpha(x)(1_{H}\ot\lambda_{t}):\ s,t\in G,\ x\in X \}\\&\sub B(H)\vt\BLT.
	\end{align*}
\end{defin}

Note that Definition \ref{def4.4} is naturally dictated by the fact that if $ M $ is a von Neumann algebra, $ \gamma $ is a $ G $-action on $ M $ and $\alpha$ is the W*-$ \LI $-action on $ M $ corresponding to $ \gamma $ as above, then the crossed product $ M\rtimes_{\alpha}G $ is equal to the normal $ \mathbb{C}1\vt L(G) $-bimodule generated by $ \alpha(M) $. This follows immediately from the well known covariance relations:
\[\alpha(\gamma_{s}(x))=(1\ot\lambda_{s})\alpha(x)(1\ot\lambda_{s}),\quad s\in G,\ x\in M. \]

\begin{remark} \label{rem9} From the discussion above, it follows that if $ (M,\alpha) $ is a W*-$ \LI $-comodule, then $ M\fcr_{\alpha}G=M\scr_{\alpha}G=M\rtimes_{\alpha}G $, where $ M\rtimes_{\alpha}G $ is the usual von Neumann algebra crossed product. Interestingly, we will prove later that this is not true in general for arbitrary $ \LI $-comodules unless $ G $ has the approximation property of Haagerup and Kraus (see Theorem \ref{thm4}).
	
Note that if $ (X,\alpha) $ is an $ \LI $-comodule with $ \alpha $ trivial, that is $ \alpha(x)=x\ot1 $ for all $x\in X$, then for $ x\in X $ and $ b\in \BLT $ we have
\begin{align*} 
   \wt{\alpha}(x\ot b)=&\ (\mathrm{id }_{X}\otimes \mathrm{Ad }U_{G}^{*})\circ(\mathrm{id }_{X}\otimes\sigma)\circ(\alpha\otimes \mathrm{id }_{B(L^{2}(G))})(x\ot b)\\
   =&\ (\mathrm{id }_{X}\otimes \mathrm{Ad }U_{G}^{*})\circ(\mathrm{id }_{X}\otimes\sigma)(x\ot 1\ot b)\\
   =&\ (\mathrm{id }_{X}\otimes \mathrm{Ad }U_{G}^{*})(x\ot b\ot 1)\\
   =&\ (\Id_{X}\ot\beta_{G})(x\ot b)
\end{align*}	
	and thus $ \wt{\alpha}=\Id_{X}\ot\beta_{G} $. Since $ \BLT^{\beta_{G}}=L(G) $ it follows that
	\begin{align*} 
	   X\fcr_{\alpha}G=&\ (X\vt\BLT)^{\wt{\alpha}}\\
	   =&\ (X\vt\BLT)^{\Id_{X}\ot\beta_{G}}\\
	   =&\ X\ft (\BLT)^{\beta_{G}}\\
	   =&\ X\ft L(G).
	\end{align*}
This actually explains the term `Fubini crossed product' which was first used in \cite{UZ}. We should note here that Hamana had already considered the notion of Fubini crossed products in \cite{Ha1} but he did not use the same term.

On the other hand, it is obvious that $ X\scr_{\alpha}G=X\vt L(G) $ when $ \alpha $ is trivial and thus the term 'spatial crossed product' is similarly justified.

Also, for a locally compact group (even in the discrete case) it is not necessarily true that $ X\vt L(G)= X\ft L(G) $ for any dual operator space $ X $. Indeed, if we take $ G $ to be any discrete group failing the approximation property (for example $G=SL(3,\mathbb{Z}) $, see \cite{LdS}), then, by \cite[Theorem 2.1]{HK}, it follows that there is a dual operator space $ X $ such that $ X\vt L(G)\neq X\ft L(G) $. Therefore, in this case, the equality $ X\fcr_{\alpha}G=X\scr_{\alpha}G $ is not valid for all $ \LI $-comodules $ (X,\alpha) $ in contrast to the von Neumann algebra case. Thus the distinction between Fubini and spatial crossed products seems to be necessary in the setting of general dual operator spaces.
\end{remark}

It was shown by Crann and Neufang \cite{CN} that if $ G $ is a locally compact group with the AP, then $ X\fcr_{\alpha}G=X\scr_{\alpha}G $ for any $ \LI $-comodule $ (X,\alpha) $ \cite[Corollary 4.8]{CN}. We warn the reader that Crann and Neufang use a different definition of the Fubini crossed product in \cite{CN}, but one can easily check that it is equivalent to Definition \ref{def4.2}. Furthermore, they consider $ G $-invariant subspaces of von Neumann algebras instead of general $ \LI $-comodules, but this is not restrictive at all, because every $ \LI $-comodule is isomorphic to a subcomodule of a W*-$ \LI $-comodule (see Remark \ref{rem1}), that is a $ G $-invariant subspace of a von Neumann algebra since every W*-$ \LI $-action comes from a pointwise $ G $-action as pointed out above.

Later, using a generalized version of Takesaki-duality, we will give an alternative proof of the aforementioned result of Crann and Neufang. Moreover, we are going to prove that its converse is also true (see Theorem \ref{thm4}). 

\begin{remark}\label{rem5} Let $ H,\ K $ be Hilbert spaces, $ X\sub B(H) $ be a w*-closed subspace and $ b,\ c\in B(K) $. Then, we have
	\[(1_{H}\ot b)(X\vt B(K))(1_{H}\otimes c)\sub X\vt B(K). \]	
	As a consequence, if $ (X,\alpha) $ is an $ \LI $-comodule, then $$ X\scr_{\alpha}G\sub X\vt \BLT, $$ because $ \alpha(X)\sub X\vt \LI\sub X\vt\BLT $. 
	
	Also, if in addition $ Y $ is a w*-closed subspace of $ B(L) $ for some Hilbert space $ L $ and $ \phi\colon X\to Y $ is a w*-continuous completely bounded map, then $ \phi\ot \mathrm{id }_{B(K)}  \colon$ $ X \vt B(K) \to Y \vt B(K) $ is a w*-continuous $ B(K) $-bimodule map in the sense that
	\[(\phi\ot \mathrm{id }_{B(K)})((1_{H}\ot a)x(1_{H}\ot b))=(1_{L}\ot a)(\phi\ot \mathrm{id }_{B(K)})(x)(1_{L}\ot b), \]
	for all $ a,\ b\in B(K) $ and $ x\in X\vt B(K) $.
\end{remark} 

For the proofs of Propositions \ref{pro2.13} and \ref{pro2.14} below the reader is referred to \cite{DA}.
\begin{pro}\cite[Proposition 3.8]{DA}\label{pro2.13} Let $ (X,\alpha) $ be an $ \LI $-comodule and suppose that $ X $ is a w*-closed subspace of $ B(H) $ for some Hilbert space $ H $. Then, $ X\fcr_{\alpha}G $ is an $ L(G) $-bimodule, i.e.
	\[(1_{H}\ot\lambda_{s})y(1_{H}\ot\lambda_{t})\in X\fcr_{\alpha}G,\qquad s,t\in G,\ y\in X\fcr_{\alpha}G \]
	and \[\alpha(X)\sub X\fcr_{\alpha}G.\]
	Therefore, we have: 
	\[X\scr_{\alpha}G\sub X\fcr_{\alpha}G. \]
	Furthermore, $ \wh{\alpha}(X\scr_{\alpha}G)\sub (X\scr_{\alpha}G)\ft L(G) $, that is $ X\scr_{\alpha}G $ is an $ L(G) $-subcomodule of $ (X\fcr_{\alpha}G,\wh{\alpha}) $.	
\end{pro}

The next result proves that, for any $ \LI $-comodule $ X $, both the Fubini crossed product and the spatial crossed product are unique up to comodule isomorphisms and thus independent of the Hilbert space on which $ X $ is represented.  	

\begin{pro}[Uniqueness of the crossed product]\cite[Proposition 3.9]{DA}\label{pro2.14} Let $ (X,\alpha) $ and $ (Y,\beta) $ be two $ \LI $-comodules and suppose that $ X $ and $ Y $ are w*-closed subspaces of $ B(H) $ and $ B(K) $ respectively. If there exists an $ \LI $-comodule isomorphism $ \Phi\colon X\to Y $, then the isomorphism $\Psi:= \Phi\ot \mathrm{id }_{\BLT}\colon X\vt\BLT\to Y\vt\BLT $ is an $ \LI $-comodule isomorphism from $ (X\vt\BLT,\wt{\alpha}) $ onto $ (Y\vt\BLT,\wt{\beta}) $, which maps $ X\fcr_{\alpha}G $ onto $ Y\fcr_{\beta}G $ and $ X\scr_{\alpha}G $ onto $ Y\scr_{\beta}G $. Also, $ \Psi|_{X\fcr_{\alpha}G} $ is an $ L(G) $-comodule isomorphism from $ (X\fcr_{\alpha}G,\wh{\alpha}) $ onto  $ (Y\fcr_{\beta}G,\wh{\beta})$ and $ \Psi|_{X\scr_{\alpha}G} $ is an $ L(G) $-comodule isomorphism from $ (X\scr_{\alpha}G,\wh{\alpha}) $ onto  $ (Y\scr_{\beta}G,\wh{\beta})$. Furthermore, $ \Psi $ is an $ L(G) $-bimodule map, i.e. $ \Psi((1_{H}\ot \lambda_{s})x(1_{H}\ot\lambda_{t}))=(1_{K}\ot \lambda_{s})\Psi(x)(1_{K}\ot\lambda_{t}) $, for all $ s,t\in G $ and $ x\in X\vt\BLT $.
\end{pro}

\begin{lem}\label{c2} Every $ \LI $-comodule is non-degenerate and saturated. In particular, for any $ \LI $-comodule $ X $ and any $ x\in X $, we have that $ x\in \overline{\LO\cdot x}^{\mathrm{w^{*}}} $.	
\end{lem}
\begin{proof} Let $ (X,\alpha) $ be an $ \LI $-comodule with $ X $ a w*-closed subspace of $ B(H) $ for some Hilbert space $ H $.
	
	By Remark \ref{rem1}, we have that $ \alpha(X) $ is an $ \LI $-subcomodule of the W*-$ \LI $-comodule $ (N,\beta) $ where $ N=B(H)\vt \LI $ and $ \beta=\Id_{B(H)}\ot\alpha_{G} $.
	
	Consider the normal *-injections $\pi_{1},\ \pi_{2}\colon N\to N\vt\LI $ given by:
	\begin{align*} 
	&\la \pi_{1}(y),\omega\ot f\ra=\int_{G}\la\Ad(1_{H}\ot\lambda_{s}^{-1})(y),\omega\ra f(s)\ ds, \\
	&\la \pi_{2}(y),\omega\ot f\ra=\int_{G}\la\Ad(1_{H}\ot\lambda_{s})(y),\omega\ra f(s)\ ds, 
	\end{align*}
	for $ y\in N $, $ \omega\in N_{*} $ and $ f\in \LO $. It is easy to verify that for any $ y\in N $ we have: 
	\begin{align*} 
	&\pi_{1}(y)=(1_{H}\ot V_{G}^{*})(y\ot 1)(1_{H}\ot V_{G})=\beta(y);\\
	&\pi_{2}(y)=(1_{H}\ot V_{G})(y\ot 1)(1_{H}\ot V_{G}^{*}).
	\end{align*}
	Therefore, since $ \pi_{1}(\alpha(X))=\beta(\alpha(X))\sub \alpha(X)\vt\LI $, it follows that for any $ s\in G $ we have
	\[\Ad(1_{H}\ot\lambda_{s}^{-1})(\alpha(X))=\alpha(X), \]
	that is \[\Ad(1_{H}\ot\lambda_{t})(\alpha(X))=\alpha(X),\ \text{for all }t\in G\]
	and thus \[\pi_{2}(\alpha(X))\sub\alpha(X)\vt\LI. \]
	
	Hence, since $ 1_{H}\ot V_{G}\in\CI\vt L(G)\vt\LI $ and $ \LI'=\LI $, it follows that both $ \Ad(1_{H}\ot V_{G}^{*}) $ and $ \Ad(1_{H}\ot V_{G}) $ map $ \alpha(X)\vt\LI $ into $ \alpha(X)\vt\LI $ and so the restriction of $ \Ad(1_{H}\ot V_{G}^{*}) $ to $ \alpha(X)\vt\LI $ is a completely isometric automorphism of $ \alpha(X)\vt\LI  $.
	
It follows from the above that the map $ \theta\colon X\vt \LI\to X\vt\LI $ defined by:
	\[\theta=(\alpha^{-1}\ot \mathrm{id}_{\LI})\circ \mathrm{Ad}(1_{H}\ot V_{G}^{*})\circ(\alpha\ot \mathrm{id}_{\LI}). \]
	is a well defined w*-continuous completely isometric automorphism of $ X\vt\LI $. Also, using the definition of $ \theta $, we get that 	
	\[\theta(x\ot1)=\alpha(x) ,\qquad  x\in X \]
	and \[\theta((1_{H}\ot f)y)=(1_{H}\ot f)\theta(y),\qquad f\in \LI,\  y\in X\vt\LI. \]	
	Therefore, it follows that
	\[X\vt\LI=\wsp\{(1_{H}\ot f)\alpha(x):\ x\in X,\ f\in\LI \}, \]
	which implies that
	\[X\vt\BLT=\wsp\{(1_{H}\ot b)\alpha(x):\ x\in X,\ b\in\BLT \}, \]
	that is $ (X,\alpha) $ is non-degenerate.
	
	So, we have proved that every $ \LI $-comodule is non-degenerate. Thus it follows (from Corollary \ref{c1}) that every $ \LI $-comodule is saturated and non-degenerate.
	
	Now, since every $ \LI $-comodule $ X $ is saturated it follows (from Proposition \ref{p8}) that $ x\in \overline{\LO\cdot x}^{\text{w*}} $ for all $ x\in X $. 	
\end{proof}

\begin{pro}\label{pro2.15} For any $ \LI $-comodule $ (X,\alpha) $, we have \[(X\fcr_{\alpha}G)^{\wh{\alpha}}=(X\scr_{\alpha}G)^{\wh{\alpha}}=\alpha(X)=\mathrm{Sat}(X,\alpha)=(X\vt\LI)^{\wt{\alpha}}. \]	
\end{pro}
\begin{proof}
	We prove first that $ \mathrm{Sat}(X,\alpha)=\left( X\vt \LI \right)^{\wt{\alpha}} $. Indeed, for any $ x\in X\vt \LI $, we have:
	\begin{align*} 
	x\in \left( X\vt \LI \right)^{\wt{\alpha}}&\iff \wt{\alpha}(x)=x\ot 1\\
	&\iff (\Id_{X}\ot \Ad U^{*}_{G})\circ (\Id_{X}\ot\sigma)\circ(\alpha\ot \Id_{\BLT})(x)=x\ot 1\\  
	&\iff(\alpha\ot \Id_{\LI})(x)=(\Id_{X}\ot\sigma)\left( (1_{H}\ot U_{G})(x\ot 1)(1_{H}\ot U^{*}_{G})\right)\\
	&\iff (\alpha\ot \Id_{\LI})(x)=(\Id_{X}\ot\alpha_{G})(x)\\
	&\iff x\in \mathrm{Sat}(X,\alpha).
	\end{align*}	
	For the fourth equivalence above we used the fact that $ \sigma\circ\alpha_{G}(f)=U_{G}(f\ot1)U_{G}^{*} $ for any $ f\in\LI $.
	
	Now, we prove that $ \left( X\fcr_{\alpha}G \right)^{\wh{\alpha}}=\left( X\vt \LI \right)^{\wt{\alpha}} $. Indeed, since the actions $ \Id_{X}\ot \delta_{G} $ and $ \wt{\alpha} $ commute (see Proposition \ref{pro4.2i}) and $ \wh{\alpha}=(\Id_{X}\ot\delta_{G})|_{X\fcr_{\alpha}G} $ it follows: 
	\begin{align*} 
	\left( X\fcr_{\alpha}G \right)^{\wh{\alpha}}&= \left( \left( X\vt \BLT \right)^{\wt{\alpha}}\right)^{\Id_{X}\ot\delta_{G}}\\
	&=\left( \left( X\vt \BLT \right)^{\Id_{X}\ot\delta_{G}}\right)^{\wt{\alpha}}\\
	&=  \left( X\ft \left( \BLT\right) ^{\delta_{G}}\right) ^{\wt{\alpha}}\\
	&=\left( X\vt \LI \right)^{\wt{\alpha}}.
	\end{align*}
	The last equality follows from the fact that $ \BLT^{\delta_{G}}=\LI $.
	
	By Lemma \ref{c2} we have that $ \mathrm{Sat}(X,\alpha)=\alpha(X) $ and thus we get
	\[(X\fcr_{\alpha}G)^{\wh{\alpha}}=\alpha(X)=\mathrm{Sat}(X,\alpha)=(X\vt\LI)^{\wt{\alpha}}.\]
	
	So it remains to show that $ (X\scr_{\alpha}G)^{\wh{\alpha}}=\alpha(X) $. Indeed, since $ (X\scr_{\alpha}G,\wh{\alpha}) $ is an $ L(G) $-subcomodule of $ (X\fcr_{\alpha}G,\wh{\alpha}) $ it follows that 
	\begin{align*} 
	   (X\scr_{\alpha}G)^{\wh{\alpha}}&=(X\fcr_{\alpha}G)^{\wh{\alpha}}\cap(X\scr_{\alpha}G)\\
	   &=\alpha(X)\cap (X\scr_{\alpha}G)\\&=\alpha(X),
	\end{align*}
	since $ \alpha(X)\sub X\scr_{\alpha}G $.
\end{proof}

\subsection{Crossed products of $ L(G) $-comodules} Here we consider the analogues of the Fubini and the spatial crossed products in the category of $ L(G) $-comodules. 

The main and most interesting difference between $ \LI $-comodules and $ L(G) $-comodules is that for any $ L(G) $-comodule $ (Y,\delta) $ the associated Fubini and spatial crossed products are equal without any further assumption on the group $ G $ or the space $ Y $ (Theorem \ref{thm1}). The reason behind this is that the Fubini and the spatial crossed product of an $ L(G) $-comodule admit a natural $ \LI $-comodule structure and thus they are always non-degenerate and saturated by Lemma \ref{c2}. This will be clear from the use of Lemma \ref{c2}  in the proof of Proposition \ref{p10} below, from which Theorem \ref{thm1} follows.

\begin{defin}\label{dtil} For an $ L(G) $-comodule $ (Y,\delta) $, we define the map \[ \widetilde{\delta} \colon Y \vt B(L^{2}(G)) \to  Y \ft B(L^{2}(G))  \ft L(G) \] by
		\[\widetilde{\delta}=(\mathrm{id }_{Y}\otimes \mathrm{Ad }W_{G})\circ(\mathrm{id }_{Y}\otimes\sigma)\circ(\delta\otimes \mathrm{id }_{B(L^{2}(G))}), \]
		where $ \sigma $ is the flip isomorphism on $\BLT\vt\BLT$.	
\end{defin}

\begin{pro}\label{pro4.2} If $ (Y,\delta) $ is an $ L(G) $-comodule, then $ \wt{\delta} $ is an $ L(G) $-action on $ Y\vt\BLT $  that commutes with the  $ \LI $-action $ \mathrm{id}_{Y}\ot\beta_{G} $.
	
\end{pro}
\begin{proof} By Remark \ref{rem1} we may suppose that $ Y$ is a w*-closed subspace of a von Neumann algebra $ N $ of the form $ N=B(H)\vt L(G) $ for some Hilbert space $ H $ and $ \delta=\varepsilon|_{Y} $, where $ \varepsilon=\Id_{B(H)}\ot\delta_{G} $. Then obviously $ \wt{\delta}=\wt{\varepsilon}|_{Y\vt\BLT} $ and $ \wt{\varepsilon} $ is a W*-$ L(G) $-action on $ N\vt\BLT $. Since $ \wt{\varepsilon} $ is a *-monomorphism, the latter can be easily verified by checking the relation \[ (\wt{\varepsilon}\ot\Id_{L(G)})\circ\wt{\varepsilon}=(\Id_{N}\ot\Id_{\BLT}\ot\delta_{G})\circ\wt{\varepsilon}\] on the generators of $ N\vt\BLT $, that is on the elements of the form $ z\ot1 $, $ 1\ot1\ot f $ and $ 1\ot1\ot\rho_{s} $ for $ z\in N $, $ f\in \LI $ and $ s\in G $, because $ \BLT $ is generated by $ R(G) $ and $ \LI $. 

Thus, in order to prove that $ \wt{\delta} $ is an $ L(G) $-action on $ Y\vt\BLT=Y\ft \BLT $, we only need to show that $ \wt{\varepsilon}(Y\ft\BLT)\sub Y\ft\BLT\ft L(G) $. Indeed, we have 
 
    \[(\varepsilon\ot\Id_{\BLT})(Y\ft\BLT)\sub Y\ft L(G)\ft \BLT\]
    and thus
   \[(\Id_{Y}\ot\sigma)\circ(\varepsilon\ot\Id_{\BLT})(Y\ft\BLT)\sub Y\ft \BLT \ft L(G).\]
Since $ W_{G}\in \LI\vt L(G) $ and $ Y\ft \BLT\ft L(G) $ is a $ \CI\vt \BLT\vt L(G)  $-bimodule, we get:
\begin{align*}
(\mathrm{id }_{Y}\otimes \mathrm{Ad }W_{G})\circ(\mathrm{id }_{Y}\otimes\sigma)\circ(\varepsilon\otimes \mathrm{id }_{B(L^{2}(G))})&(Y\ft\BLT)\\&\sub Y\ft\BLT\ft L(G).
\end{align*}
 
 On the other hand, in order to prove that $ \wt{\delta} $ and $ \Id_{Y}\ot\beta_{G} $ commute,  it suffices to verify that $ \mathrm{id}_{B(H)}\ot \mathrm{id}_{L(G)}\ot\beta_{G} $ and $ \wt{\varepsilon} $ commute, where $ \varepsilon=\mathrm{id}_{B(H)}\ot \delta_{G} $. Because $ \wt{\varepsilon} $ and $ \mathrm{id}_{B(H)}\ot \mathrm{id}_{L(G)}\ot\beta_{G} $ act identically on the first factor $ B(H) $, we only need to prove that $   \mathrm{id}_{L(G)}\ot\beta_{G} $ and $ \wt{\delta_{G}} $ commute, that is:
\begin{align}\label{1}
&(\wt{\delta_{G}}\ot \mathrm{id}_{\LI})\circ(\mathrm{id}_{L(G)}\ot \beta_{G} )=\\
=&(\mathrm{id}_{L(G)} \ot \mathrm{id}_{\BLT} \ot \sigma)\circ(\mathrm{id}_{L(G)} \ot \beta_{G}\ot \mathrm{id}_{L(G)} )\circ\wt{\delta_{G}}\nonumber
\end{align}
Let $ S $ denote the unitary on $ \LT\ot \LT $ with $ S(\xi\ot\eta)=\eta\ot\xi $. Thus, the flip isomorphism $ \sigma $ on $ \BLT\vt\BLT $ is written as $ \sigma= \mathrm{Ad}S $.
If $ a\in L(G) $ and $ b\in\BLT $, then by applying the left and right hand sides of \eqref{1} on $ a\ot b $, we get respectively:
\begin{align}\label{2} 
&(\wt{\delta_{G}}\ot \mathrm{id}_{\LI})\circ(\mathrm{id}_{L(G)}\ot \beta_{G} )(a\ot b)=\nonumber\\
&\mathrm{Ad}[(1\ot W_{G}\ot 1)(1\ot S\ot1)(W_{G}^{*}\ot1\ot1)(1\ot S\ot1)(1\ot1\ot S)\\
&(1\ot U_{G}^{*}\ot1) ](a\ot b\ot 1\ot 1)\nonumber 
\end{align}

and

\begin{align}\label{3} 
&(\mathrm{id}_{L(G)} \ot \mathrm{id}_{\BLT} \ot \sigma)\circ(\mathrm{id}_{L(G)} \ot \beta_{G}\ot \mathrm{id}_{L(G)} )\circ\wt{\delta_{G}}(a\ot b)=\nonumber\\
&\mathrm{Ad}[(1\ot 1\ot S)(1\ot U_{G}^{*}\ot1)(1\ot1\ot S)(1\ot W_{G}\ot1)(1\ot S\ot 1)\\
&( W_{G}^{*}\ot1\ot1)(1\ot S\ot 1) ](a\ot b\ot 1\ot 1)\nonumber 
\end{align}

Consider the unitaries in the square brackets in \eqref{2} and \eqref{3}: \[A=(1\ot W_{G}\ot 1)(1\ot S\ot1)(W_{G}^{*}\ot1\ot1)(1\ot S\ot1)(1\ot1\ot S)(1\ot U_{G}^{*}\ot1)\] and  \[B=(1\ot 1\ot S)(1\ot U_{G}^{*}\ot1)(1\ot1\ot S)(1\ot W_{G}\ot1)(1\ot S\ot 1)( W_{G}^{*}\ot1\ot1)(1\ot S\ot 1).\]
Then, \eqref{1} is equivalent to 
\[A(a\ot b\ot 1\ot 1)A^{*}=B(a\ot b\ot 1\ot 1)B^{*},\quad\text{for all } a\in L(G)\text{ and } b\in \BLT, \]
which in turn is equivalent to the condition: 
\[ A^{*}B\in R(G)\vt \mathbb{C}1\vt\BLT\vt\BLT, \]
which is true since a computation shows that \[A^{*}B\in \mathbb{C}1 \vt \mathbb{C}1\vt\BLT\vt\BLT.\]	
\end{proof}

\begin{defin}\label{def4.2i} Let $ (Y,\delta) $ be an $L(G) $-comodule. The \bfc{Fubini crossed product} of $ Y $ by $ \delta $ is defined to be the $ \LI $-comodule $ (Y\ltimes^{\cl{F}}_{\delta}G,\ \widehat{\delta}) $, where
		\[Y\ltimes^{\cl{F}}_{\delta}G:=(Y\vt B(L^{2}(G)))^{\widetilde{\delta}} \]
		and \[\widehat{\delta}:=(\mathrm{id }_{X}\otimes\beta_{G})|_{Y\dfcr_{\delta}G}. \]
		The $ \LI $-action $ \widehat{\delta}\colon Y\dfcr_{\delta}G\to (Y\dfcr_{\delta}G)\vt \LI $ is called the \bfc{dual action} of $ \delta $. 
\end{defin} 

By Proposition \ref{pro4.2} and Lemma \ref{lem2.3},  $ (Y\dfcr_{\delta},\wh{\delta}) $ is indeed an $ \LI $-subcomodule of $ (Y\vt\BLT,\Id_{Y}\ot\beta_{G}) $.

\begin{defin} \label{def2.11}  Let $ (Y,\delta) $ be an $ L(G) $-comodule and suppose that $ Y $ is w*-closed in $ B(K) $ for some Hilbert space $ K $. The \bfc{spatial crossed product} of $ Y $ by $ \delta $ is defined to be the space
		\begin{align*}
			Y\dscr_{\delta}G&:=\wsp\{(1_{K}\ot f)\delta(y)(1_{K}\ot g):\ f,g\in\LI,\ y\in Y \}\\&\sub B(K)\vt\BLT.
		\end{align*} 
\end{defin} 

\begin{pro}\label{pro2.13i} Let $ (Y,\delta) $ be an $ L(G) $-comodule and suppose that $ Y $ is a w*-closed subspace of $ B(K) $ for some Hilbert space $ K $. Then, $ Y\dfcr_{\delta}G $ is an $ L^{\ap}(G) $-bimodule, i.e.
	\[(1_{K}\ot f)y(1_{K}\ot g)\in Y\dfcr_{\delta}G,\qquad f,g\in \LI,\ y\in Y\dfcr_{\delta}G \]
	and $ \delta(Y)\sub Y\dfcr_{\delta}G $. Thus, we have: 
	\[Y\dscr_{\delta}G\sub Y\dfcr_{\delta}G. \]
	In addition, $ \wh{\delta}(Y\dscr_{\delta}G)\sub (Y\dscr_{\delta}G)\vt \LI $, that is $ Y\dscr_{\delta}G $ is an $ \LI $-subcomodule of $ (Y\dfcr_{\delta}G,\wh{\delta}) $.	
\end{pro}
\begin{proof} Let $ f\in\LI $ and $ y\in  Y\dfcr_{\delta}G  $. Then, by Remark \ref{rem5} we have that $ (1_{K}\ot f)y\in Y\vt \BLT $ and $ \wt{\delta}(y)=y\ot 1 $, by Definition \ref{def4.2i}. Also, by Remark \ref{rem5}, we have that
	\[(\delta\ot \mathrm{id }_{\BLT})((1_{K}\ot f)y)=(1_{K}\ot1_{\LT}\ot f)(\delta\ot \mathrm{id }_{\BLT})(y). \] 
	Thus, it follows:
	\begin{align*} 
	\wt{\delta}((1_{H}\ot f)y)&=(\mathrm{id }_{Y}\otimes \mathrm{Ad }W_{G})\circ(\mathrm{id }_{Y}\otimes\sigma)\circ(\delta\otimes \mathrm{id }_{B(L^{2}(G))})((1_{K}\ot f)y)\\
	&=(\mathrm{id }_{Y}\otimes \mathrm{Ad }W_{G})\circ(\mathrm{id }_{Y}\otimes\sigma)\left((1_{K}\ot1_{\LT}\ot f)(\delta\ot \mathrm{id }_{\BLT})(y) \right)\\
	&=\left[ (\mathrm{id }_{B(K)}\otimes \mathrm{Ad }W_{G})\circ(\mathrm{id }_{B(K)}\otimes\sigma)((1_{K}\ot1_{\LT}\ot f))\right] \wt{\delta}(y)\\
	&=\left[(1_{K}\ot W_{G})(1_{K}\ot f\ot 1_{\LT})(1_{K}\ot W^{*}_{G}) \right](y\ot1_{\LT})\\
	&=(1_{K}\ot f)y\ot1_{\LT}, 
	\end{align*}
	where the third equality above holds since $ (\mathrm{id }_{B(K)}\otimes \mathrm{Ad }W_{G})\circ(\mathrm{id }_{B(K)}\otimes\sigma) $ is a *-homomorphism and thus multiplicative, while the last equality is true because $ W_{G}\in \LI\vt L(G) $ and thus $ W_{G}(f\ot1)W_{G}^{*}=f\ot1 $, and $ \wt{\delta}(y)=y\ot 1 $. Therefore, $$ (1_{K}\ot f)y\in Y\dfcr_{\delta}G.$$ Similarly, we get $ y(1_{K}\ot g)\in Y\dfcr_{\delta}G $ for all $g\in \LI  $ and $ y\in Y\dfcr_{\delta}G $.
	
	On the other hand, if $ x\in Y $, then:
	\begin{align*} 
	\wt{\delta}(\delta(x))&=(\mathrm{id }_{Y}\otimes \mathrm{Ad }W_{G})\circ(\mathrm{id }_{Y}\otimes\sigma)\circ(\delta\otimes \mathrm{id }_{B(L^{2}(G))})(\delta(x))\\
	&=(\mathrm{id }_{Y}\otimes \mathrm{Ad }W_{G})\circ(\mathrm{id }_{Y}\otimes\sigma)\circ(\mathrm{id }_{Y}\ot\delta_{G})(\delta(x))\\
	&=	(\mathrm{id }_{Y}\otimes \mathrm{Ad }W_{G})\circ(\mathrm{id }_{Y}\ot\delta_{G})(\delta(x))\\
	&=(1_{K}\ot W_{G})(1_{K}\ot W_{G}^{*})(\delta(x)\ot 1_{\LT})(1_{K}\ot W_{G})(1_{K}\ot W_{G}^{*})\\
	&=\delta(x)\ot 1_{\LT},
	\end{align*}
	where the third equality holds because $ \sigma\circ\delta_{G}=\delta_{G} $. Hence, $ \delta(Y)\sub Y\dfcr_{\delta}G $.
	
	Let $ x\in Y $ and $ f\in \LI $. Since $ \beta_{G}(z)=z\ot 1 $ for any $ z\in L(G) $ and $ \delta(x)\in Y\ft L(G) $, it follows that $(\mathrm{id }_{B(K)}\ot\beta_{G})(\delta(x))=\delta(x)\ot 1$. Thus we get:
	\begin{align*} 
	\wh{\delta}((1_{K}\ot f)\delta(x))&=(\mathrm{id }_{B(K)}\ot\beta_{G})((1_{K}\ot f)\delta(x))\\
	&=(1_{K}\ot\beta_{G}(f))(\mathrm{id }_{B(K)}\ot\beta_{G})(\delta(x))\\
	&=(1_{K}\ot\beta_{G}(f))(\delta(x)\ot 1)\in (Y\dscr_{\delta}G)\vt \LI,
	\end{align*}
	because $ \beta_{G}(f)\in \LI\vt \LI $ and $\delta(x)\in Y\dscr_{\delta}G$. Therefore $ (Y\dscr_{\delta}G,\wh{\delta} )$ is  an $ \LI $-subcomodule of $(Y\dfcr_{\delta}G,\wh{\delta} ) $.	
\end{proof}

\begin{pro}[Uniqueness of the crossed product]\label{pro2.14i} Let $ (Y,\delta) $ and $ (Z,\varepsilon) $ be two $ L(G) $-comodules and suppose that $ Y $ and $ Z $ are w*-closed subspaces of $ B(H) $ and $ B(K) $ respectively. If there exists an $ L(G) $-comodule isomorphism $ \Phi\colon Y\to Z $, then the isomorphism $\Psi:= \Phi\ot \mathrm{id }_{\BLT}\colon Y\vt\BLT\to Z\vt\BLT $ is an $ L(G) $-comodule isomorphism from $ (Y\vt\BLT,\wt{\delta}) $ onto $ (Z\vt\BLT,\wt{\varepsilon}) $, which maps $ Y\dfcr_{\delta}G $ onto $ Z\dfcr_{\varepsilon}G $ and $ Y\dscr_{\delta}G $ onto $ Z\scr_{\varepsilon}G $. Also, $ \Psi|_{Y\dfcr_{\delta}G} $ is an $ \LI $-comodule isomorphism from $ (Y\dfcr_{\delta}G,\wh{\delta}) $ onto  $ (Z\dfcr_{\varepsilon}G,\wh{\varepsilon})$ and $ \Psi|_{Y\dscr_{\delta}G} $ is an $ \LI $-comodule isomorphism from $ (Z\dscr_{\delta}G,\wh{\delta}) $ onto  $ (Z\dscr_{\varepsilon}G,\wh{\varepsilon})$. Furthermore, $ \Psi $ is an $ \LI $-bimodule map, i.e. $ \Psi((1_{H}\ot f)x(1_{H}\ot g))=(1_{K}\ot f)\Psi(x)(1_{K}\ot g) $, for all $ f,g\in \LI $ and $ x\in Y\vt\BLT $.
\end{pro}
\begin{proof} First, since $ \Phi $ is a comodule morphism we have that $ \varepsilon\circ\Phi=(\Phi\ot \Id)\circ\delta $ and hence:
	\begin{align*} 
	\wt{\varepsilon}\circ\Psi&=(\Id\ot \mathrm{Ad }W_{G})\circ(\Id\ot\sigma)\circ(\varepsilon\ot \Id)\circ(\Phi\ot \Id)\\ 
	&=(\Id\ot \mathrm{Ad }W_{G})\circ(\Id\ot\sigma)\circ((\varepsilon\circ\Phi)\ot \Id)\\
	&=(\Id\ot \mathrm{Ad }W_{G})\circ(\Id\ot\sigma)\circ\left[ ((\Phi\ot \Id)\circ\delta)\ot \Id\right] \\
	&=(\Id\ot \mathrm{Ad }W_{G})\circ(\Id\ot\sigma)\circ(\Phi\ot \Id\ot \Id)\circ(\delta\ot \Id)\\
	&=(\Phi\ot \Id\ot \Id)\circ(\Id\ot \mathrm{Ad }W_{G})\circ(\Id\ot\sigma)\circ(\delta\ot \Id)\\
	&=(\Psi\ot \Id)\circ\wt{\delta},
	\end{align*}
	which proves that $ \Psi $ is an $ L(G) $-comodule isomorphism from $ (Y\vt\BLT,\wt{\delta}) $ onto $ (Z\vt\BLT,\wt{\varepsilon}) $. This implies that $ \Psi $ maps the fixed point subspace $ Y\dfcr_{\delta}G $ of $ \wt{\delta} $ onto the fixed point subspace $ Z\dfcr_{\varepsilon}G $ of $ \wt{\varepsilon} $. 
	
	On the other hand, the relation $ \varepsilon\circ\Phi=(\Phi\ot \Id)\circ\delta $ yields that
	\[\Psi(\delta(Y))=(\Phi\ot \Id)(\delta(Y))=\varepsilon(\Phi(Y))=\varepsilon(Z) \]
	and since $ \Psi $ is an $ \LI $-bimodule isomorphism (see Remark \ref{rem5}) it follows that $ \Psi $ maps $ Y\dscr_{\delta}G $ onto $ Z\dscr_{\varepsilon}G $.
	
	 Finally, we have
	\begin{align*} 
	\wh{\varepsilon}\circ\Psi&=(\mathrm{id }_{Z}\ot\beta_{G})\circ(\Phi\ot \mathrm{id }_{\BLT})\\&=(\Phi\ot \mathrm{id }_{\BLT}\ot \mathrm{id }_{\LI})\circ(\mathrm{id }_{Y}\ot\beta_{G})\\
	&=(\Psi\ot \mathrm{id }_{\LI})\circ\wh{\delta}.
	\end{align*}	
\end{proof}

Note that until now everything seems to work in complete analogy to the case of $ \LI $-comodules. However, from now on the differences between $ \LI $-comodules and $L(G)$-comodules will start to become apparent.

\begin{pro}\label{p11} For any $ L(G) $-comodule $ (Y,\delta) $ we have:
	\[\delta(Y)\sub\left( Y\dfcr_{\delta}G \right)^{\wh{\delta}}=\mathrm{Sat}(Y,\delta)=\left( Y\ft L(G) \right)^{\wt{\delta}}.\]
\end{pro}
\begin{proof} Since $ \delta(Y)\sub \mathrm{Sat}(Y,\delta) $ is obvious (see Definition \ref{d2}) we only have to show the equalities $  \left( Y\dfcr_{\delta}G \right)^{\wh{\delta}}=\mathrm{Sat}(Y,\delta)=\left( Y\ft L(G) \right)^{\wt{\delta}}$. Suppose that $ Y $ is a w*-closed subspace of $ B(H) $ for some Hilbert space $ H $.
	
	We prove first that $ \mathrm{Sat}(Y,\delta)=\left( Y\ft L(G) \right)^{\wt{\delta}} $. Indeed, for any $ x\in Y\ft L(G) $, we have:
	\begin{align*} 
	x\in \left( Y\ft L(G) \right)^{\wt{\delta}}&\iff \wt{\delta}(x)=x\ot 1\\
	&\iff (\Id_{Y}\ot \Ad W_{G})\circ (\Id_{Y}\ot\sigma)\circ(\delta\ot \Id_{\BLT})(x)=x\ot 1\\  
	&\iff(\delta\ot \Id_{L(G)})(x)=(\Id_{Y}\ot\sigma)\left( (1_{H}\ot W_{G}^{*})(x\ot 1)(1_{H}\ot W_{G})\right)\\
	&\iff  (\delta\ot \Id_{L(G)})(x)=(\Id_{Y}\ot\sigma)\circ(\Id_{Y}\ot\delta_{G})(x)\\
	&\iff (\delta\ot \Id_{L(G)})(x)=(\Id_{Y}\ot\delta_{G})(x)\\
	&\iff x\in \mathrm{Sat}(Y,\delta),
	\end{align*}	
	where for the fourth equivalence above we used the fact that $ \sigma\circ\delta_{G}=\delta_{G} $ since $ \delta_{G}(\lambda_{s})=\lambda_{s}\ot\lambda_{s} $ for all $ s\in G $.
	
	It remains to prove that $ \left( Y\dfcr_{\delta}G \right)^{\wh{\delta}}=\left( Y\ft L(G) \right)^{\wt{\delta}} $. Indeed, since the actions $ \Id_{Y}\ot \beta_{G} $ and $ \wt{\delta} $ commute (see Proposition \ref{pro4.2}) and $ \wh{\delta}=(\Id_{Y}\ot\beta_{G})|_{Y\dfcr_{\delta}G} $ it follows: 
	\begin{align*} 
	\left( Y\dfcr_{\delta}G \right)^{\wh{\delta}}&= \left( \left( Y\vt \BLT \right)^{\wt{\delta}}\right)^{\Id_{Y}\ot\beta_{G}}\\
	&=\left( \left( Y\vt \BLT \right)^{\Id_{Y}\ot\beta_{G}}\right)^{\wt{\delta}}\\
	&=  \left( Y\ft \left( \BLT\right) ^{\beta_{G}}\right) ^{\wt{\delta}}\\
	&=\left( Y\ft L(G) \right)^{\wt{\delta}}.
	\end{align*}
	The last equality follows from the fact that $ \BLT^{\beta_{G}}=L(G) $.	
\end{proof}

\begin{remark}\label{remW} There is a selfadjoint unitary operator $ \Lambda \in\BLT $, such that
	\[\Lambda\rho_{t}\Lambda=\lambda_{t},\qquad t\in G, \]
	namely \[\Lambda\xi(s)=\Delta_{G}(s)^{-1/2}\xi(s^{-1}),\qquad s\in G,\ \xi\in\LT. \]
	We put  \[W_{\Lambda}:=(1\ot \Lambda)W_{G}, \] that is  \[W_{\Lambda}\xi(s,t)=\Delta_{G}(t)^{-1/2}\xi(s,st^{-1})\qquad s,t\in G,\ \xi\in L^{2}(G\times G). \] 
	
	It is easy to verify that \[U_{G}W_{\Lambda}S=W_{G},\]  where $ S\xi(s,t)=\xi(t,s) $ is the flip operator on $ L^{2}(G\times G) $. Also, $$ W_{\Lambda}\in \LI\vt\BLT, $$ because $ W_{G}\in \LI\vt L(G) $.	
\end{remark}

\begin{pro}\label{p10} Let $ (Y,\delta) $ be an $ L(G) $-comodule  and suppose that $ Y $ is a w*-closed subspace of $ B(H) $ for some Hilbert space $ H $. Then we have:
	\[Y\dfcr_{\delta}G=\wsp\left\lbrace (\mathbb{C}1_{H}\vt\LI)\left(Y\dfcr_{\delta}G\right)^{\wh{\delta}} \right\rbrace . \]	
\end{pro}	
\begin{proof} First put $ K:=H\ot\LT $ and $ X:=Y\dfcr_{\delta}G $. Then $ (X,\wh{\delta}) $ is an $ \LI $-subcomodule of $ (B(K),\alpha) $, where $ \alpha:=\mathrm{id}_{B(H)}\ot \beta_{G}\colon B(K)\to B(K)\vt\LI $.
	
	Consider the $ \LI $-actions \[\wt{\alpha},\ \bar{\alpha} \colon  B(K)\vt\BLT\to B(K)\vt \BLT\vt\LI\]
	defined by
	\[\wt{\alpha}=( \Id _{B(K)} \ot \Ad U_{G}^{*})\circ (\Id_{B(K)}\ot\sigma)\circ(\alpha\ot\Id_{\BLT}) \]
	and \[\bar{\alpha}= (\Id_{B(K)}\ot\sigma)\circ(\alpha\ot\Id_{\BLT}). \]	
	Recall the unitary $ W_{\Lambda} $ with $ W_{\Lambda}\xi(s,t)=\Delta_{G}(t)^{-1/2}\xi(s,st^{-1}) $ and put  \[W:=1_{H}\ot W_{\Lambda}. \] 
	Since $ W_{\Lambda}\in \LI\vt \BLT$, it follows that  \[W\in \mathbb{C}1_{H}\vt\LI\vt\BLT\sub B(K)\vt\BLT.\] 
	
	\paragraph*{\textbf{Claim:}} The normal *-automorphism \[\Ad W\colon B(K)\vt\BLT\to B(K)\vt\BLT\] is an $ \LI $-comodule isomorphism from $ (B(K)\vt\BLT,\wt{\alpha}) $ onto\\ $ (B(K)\vt\BLT,\bar{\alpha}) $, that is:
	\begin{equation}\label{eq4} 
	\bar{\alpha}\circ\Ad W=(\Ad W\ot\Id_{\LI})\circ\wt{\alpha}.
	\end{equation}
	
	\paragraph*{\emph{Proof of the Claim:}} In order to prove \eqref{eq4} we show first the following 
	\begin{equation}\label{eq5} 
	\bar{\alpha}(W)=(W\otimes1_{\LT})(1_{K}\otimes U_{G}^{*}).
	\end{equation}
	Let $ S\in \BLT\vt\BLT $ denote the flip operator, i.e. $ S(\xi\ot\eta)=\eta\ot\xi $ and thus $ \Ad S=\sigma $. For any $ a\in B(H) $ and $ b, c\in \BLT $ we have:
	\begin{align*} 
	\bar{\alpha}(a\ot  b\ot c)=&\ (\Id_{B(K)}\ot \sigma)(\alpha(a\ot b)\ot c)\\
	=&\ (\Id_{B(K)}\ot \sigma)(a\ot \beta_{G}(b)\ot c)\\
	=&\ (\Id_{B(K)}\ot \sigma)(a\ot (U_{G}^{*}(b\ot 1)U_{G})\ot c)\\
	=&\ (1\ot 1\ot S)(1\ot U_{G}^{*}\ot1)(a\ot b\ot 1\ot c)(1\ot U_{G}\ot1)(1\ot 1\ot S)\\
	=&\ (1\ot 1\ot S)(1\ot U_{G}^{*}\ot1)(1\ot 1\ot S)(a\ot b\ot c\ot 1)(1\ot 1\ot S)\\
	&\ (1\ot U_{G}\ot1)(1\ot 1\ot S).
	\end{align*}  
	It follows that 
	\begin{align*} 
	\bar{\alpha}(W)=&(1\ot 1\ot S)(1\ot U_{G}^{*}\ot1)(1\ot 1\ot S)(W\ot 1)(1\ot 1\ot S)\\
	&(1\ot U_{G}\ot1)(1\ot 1\ot S)\\
	=&(1\ot 1\ot S)(1\ot U_{G}^{*}\ot1)(1\ot 1\ot S)(1\ot W_{\Lambda}\ot 1)(1\ot 1\ot S)\\
	&(1\ot U_{G}\ot1)(1\ot 1\ot S)\\
	\end{align*}
	and therefore \eqref{eq5} is equivalent to the following
	\[ ( 1\ot S)( U_{G}^{*}\ot1)( 1\ot S)( W_{\Lambda}\ot 1)( 1\ot S)( U_{G}\ot1)( 1\ot S)=( W_{\Lambda}\ot 1)(1\ot U_{G}^{*}), \]
	which can be easily checked by computation. Thus \eqref{eq5} is proved.
	
	Now \eqref{eq4} follows immediately from \eqref{eq5} since, for any $ T\in B(K)\vt\BLT $, we have
	\begin{align*} 
	(\bar{\alpha}\circ \mathrm{Ad }W)(T)&=\bar{\alpha}(W)\bar{\alpha}(T)\bar{\alpha}(W)^{*}\\&=(W\otimes1)(1\otimes U_{G}^{*})\bar{\alpha}(T)(1\otimes U_{G})(W^{*}\otimes 1)\\
	&=(\mathrm{Ad }W\otimes \mathrm{id }_{\LI})\circ(\mathrm{id }_{B(K)}\otimes \mathrm{Ad }U^{*}_{G})\circ\bar{\alpha}(T)\\
	&=(\mathrm{Ad }W\otimes \mathrm{id }_{\LI})\circ\wt{\alpha}(T). 
	\end{align*}
	Thus the Claim is proved.
	
	By Corollary \ref{c2}, $ (X,\alpha) $ is non-degenerate, that is:
	\[X\vt \BLT=\wsp\{(\mathbb{C}1_{K}\vt\BLT)\alpha(X) \} \] 
	and therefore we get:
	\begin{align*} 
	X\vt \BLT&\sub\wsp\{(\mathbb{C}1_{K}\vt\LI)(\mathbb{C}1_{K}\vt L(G))\alpha(X) \}\\
	&\sub\wsp\{(\mathbb{C}1_{K}\vt\LI)(X\fcr_{\alpha}G) \},
	\end{align*}
	because  $ \BLT=\wsp\{\LI L(G) \} $ and $ (\mathbb{C}1_{K}\vt L(G))\alpha(X)\sub X\scr_{\alpha}G$ $\sub X\fcr_{\alpha}G $.  
	
	Since $ W\in \mathbb{C}1_{H}\vt\LI\vt\BLT $ and $ X $ is a $ \mathbb{C}1_{H}\vt\LI $-module, we have that $ \Ad W $ maps $ X\vt\BLT $ onto itself and therefore we get:
	\begin{align*} X\vt\BLT&=W(X\vt\BLT)W^{*}\\&\sub\wsp\{W(\mathbb{C}1_{K}\vt\LI)W^{*}W(X\fcr_{\alpha}G)W^{*} \}. \end{align*}	
	
	Also, $ X\vt \BLT $ is an $ \LI $-subcomodule of both $ (B(K)\vt\BLT,\wt{\alpha}) $ and $ (B(K)\vt\BLT,\bar{\alpha})  $ and thus it follows from \eqref{eq4} and the \textbf{Claim} that the restriction of $ \Ad W $ to $ X\vt\BLT $ is an $ \LI $-comodule isomorphism from $ (X\vt\BLT,\wt{\alpha}) $ onto $ (X\vt\BLT,\bar{\alpha}) $. Therefore, it maps the fixed point subspace $ X\fcr_{\alpha}G=(X\vt\BLT)^{\wt{\alpha}} $ onto $ (X\vt\BLT)^{\bar{\alpha}}=X^{\alpha}\vt\BLT $.
	
	On the other hand, we have  \[W(\mathbb{C}1_{K}\vt\LI)W^{*} \sub \mathbb{C}1_{H} \vt \LI \vt \BLT,\] since $ W\in\mathbb{C}1_{H} \vt\LI$ $\vt\BLT $.	
	Therefore, it follows that: 
	\begin{align*} 
	X\vt\BLT&\sub\wsp\{(\mathbb{C}1_{H} \vt\LI\vt \BLT)(X^{\alpha}\vt\BLT) \}\\
	&\sub\left( \wsp\{(\mathbb{C}1_{H} \vt\LI)X^{\alpha}\}\right) \vt \BLT.
	\end{align*}
	Since the reverse inclusion  is obvious we get that \[ X\vt\BLT=\left( \wsp\{(\mathbb{C}1_{H} \vt\LI)X^{\alpha}\}\right) \vt \BLT \] and therefore $ X=\wsp\{(\mathbb{C}1_{H} \vt\LI)X^{\alpha}\} $, that is:
	\[ Y\dfcr_{\delta}G=\wsp\left\lbrace (\mathbb{C}1_{H}\vt\LI)\left(Y\dfcr_{\delta}G\right)^{\wh{\delta}} \right\rbrace .\]
\end{proof}

\begin{thm}\label{thm1} For every $ L(G) $-comodule $ (Y,\delta) $ we have:
	\[Y\dfcr_{\delta}G=Y\dscr_{\delta}G. \]	
\end{thm} 
\begin{proof} Since $ Y\dscr_{\delta}G\sub Y\dfcr_{\delta}G $ it suffices to prove that $ Y\dfcr_{\delta}G\sub Y\dscr_{\delta}G $. Suppose that $ Y $ is a w*-closed subspace of $ B(H) $ for some Hilbert space $ H $ and consider the following W*-$ L(G) $-action on $ B(H)\vt\BLT $:
	
	\[\varepsilon:=\Id_{B(H)}\ot\delta_{G}\colon B(H)\vt\BLT\to B(H)\vt\BLT\vt L(G)\]
	that is
	 \[\varepsilon(x)=(1_{H}\ot W_{G}^{*})(x\ot1)(1_{H}\ot W_{G}),\qquad x\in  B(H)\vt\BLT\]
	
	For any $ x\in \mathrm{Sat}(Y,\delta) $ and $ f\in\LI $ we have:
	\begin{align*} 
	\varepsilon((1_{H}\ot f)x)&=(\Id_{B(H)}\ot \delta_{G})((1_{H}\ot f)x)\\
	&=(1_{H}\ot f\ot 1_{\LT})(\Id_{B(H)}\ot \delta_{G})(x)\\
	&=(1_{H}\ot f\ot 1_{\LT})(\delta\ot \Id_{L(G)})(x),
	\end{align*}
	where the last equality is obtained by the definition of $ \mathrm{Sat}(Y,\delta) $. Since $ \mathrm{Sat}(Y,\delta)\sub Y\ft L(G) $, it follows that $ (\delta\ot \Id_{L(G)})(x)\in \delta(Y)\ft L(G) \sub \left( Y\dscr_{\delta}G \right)\ft L(G) $ and thus $ (1_{H}\ot f\ot 1_{\LT})(\delta\ot \Id_{L(G)})(x)\in \left( Y\dscr_{\delta}G \right)\ft L(G) $, because  $ \left( Y\dscr_{\delta}G \right)\ft L(G)  $ is  a $ \mathbb{C}1_{H}\vt \LI\vt L(G) $-bimodule since $ Y\dscr_{\delta}G $ is a $ \mathbb{C}1_{H}\vt \LI $-bimodule.
	
	 Thus we have proved that $ \varepsilon $ maps $ \wsp\{ (\mathbb{C}1_{H}\vt \LI)\mathrm{Sat}(Y,\delta)\}	 $ into the Fubini tensor product $ \left( Y\dscr_{\delta}G \right)\ft L(G) $ and so Proposition \ref{p10} and Proposition \ref{p11} imply that 
	\[\varepsilon\left(Y\dfcr_{\delta}G\right)\sub\left( Y\dscr_{\delta}G \right)\ft L(G) \]
	that is \[\left( Y\dfcr_{\delta}G\right) \vt\mathbb{C}1\sub(1_{H}\ot W_{G})\left( (Y\dscr_{\delta}G) \ft L(G)\right)(1_{H}\ot W_{G}^{*}). \] 
	
	On the other hand, $ 1_{H}\ot W_{G}\in \mathbb{C}1_{H}\vt \LI\vt L(G) $ and $ (Y\dscr_{\delta}G) \ft L(G) $ is a $ \mathbb{C}1_{H}\vt \LI\vt L(G) $-bimodule. Therefore we get that \[\left( Y\dfcr_{\delta}G\right) \vt\mathbb{C}1\sub (Y\dscr_{\delta}G) \ft L(G)\] and thus $ Y\dfcr_{\delta}G\sub Y\dscr_{\delta}G $. 	 
\end{proof}  

Theorem \ref{thm1} allows us to give the next definition in order to simplify our notation:
 
\begin{defin} For an $ L(G) $-comodule $ (Y,\delta) $, we will write $ Y\lt_{\delta}G $ instead of $ Y\dfcr_{\delta}G $ or $ Y\dscr_{\delta}G $.
\end{defin}

\subsection{The approximation property}

Note that the key to the proof of Proposition \ref{p10}, which yields Theorem \ref{thm1}, is that every $ \LI $-comodule is non-degenerate (see Lemma \ref{c2}) and thus $ (Y\dfcr_{\delta} G,\wh{\delta}) $ is non-degenerate. 

If the analogue of Lemma \ref{c2} for $ L(G) $-comodules were also valid, we would be able to prove the analogue of Proposition \ref{p10} for $ \LI $-comodules and consequently  that $ X\fcr_{\alpha}G=X\scr_{\alpha}G $ for any $ \LI $-comodule $ (X,\alpha) $.  However, this is not the case for every group $ G $. In fact, we will prove that every $ L(G) $-comodule is non-degenerate if and only if $ G $ has the approximation property of U. Haagerup and J. Kraus \cite{HK}. To this end we will need some preparation.

Following \cite{HK}, a complex-valued function $ u\colon G\to \mathbb{C} $ is called a \bfc{multiplier} for the Fourier algebra $ A(G) $ if the linear map $  m_{u}(v)=uv $
maps $ A(G) $ into $ A(G) $. For a multiplier $ u $ we denote by $ M_{u}\colon L(G)\to L(G) $ the dual  of $ m_{u} $. The function $ u $ is called a \bfc{completely bounded multiplier} if $ M_{u} $ is completely bounded. The space of all completely bounded multipliers is denoted by $ M_{0}A(G) $ and it is a Banach space with the norm $ ||u||_{M_{0}}=||M_{u}||_{cb} $. Moreover, $ A(G)\sub M_{0}A(G) $.

It is known that $ M_{0}A(G) $ is the dual Banach space of a certain completion $ Q(G) $ of $ \LO $.

We say that $ G $ has \bfc{the approximation property} (shortly AP) if there is a net $ \{u_{i}\}_{i\in I} $ in $ A(G) $, such that $ u_{i}\rightarrow1 $ in the $ \sigma(M_{0}A(G),Q(G)) $-topology.

For more details on completely bounded multipliers and the approximation property see for example  \cite{dCH}, \cite{CH} and \cite{HK}.

We will need the following theorem due to U. Haagerup and J. Kraus (see \cite[Proposition 1.7 and Theorem 1.9]{HK}). 

\begin{thm}[Haagerup-Kraus]\label{thm7.1} For a locally compact group $ G $, the following conditions are equivalent:
	\begin{itemize}
		\item[(i)] $ G $ has the AP;
		\item[(ii)] There is a net $ \{u_{i}\} $ in $ A(G) $ such that, for any von Neumann algebra $ N $, $ (\mathrm{id }_{N}\ot M_{u_{i}})(x)\longrightarrow x$ in the w*-topology for all $ x\in N\vt L(G) $;
		\item [(iii)] For some separable infinite dimensional Hilbert space $ K $ there is a net $ \{u_{i}\} $ in $ A(G) $ such that $ (\mathrm{id }_{B(K)}\ot M_{u_{i}})(x)\longrightarrow x$ in the w*-topology for all $ x\in B(K)\vt L(G) $.
	\end{itemize}	
\end{thm} 

\begin{remark}\label{cor6.1} It was shown in \cite[Corollary II.1.5]{SVZ2} that if $ \delta\colon N\to N\vt L(G) $ is a W*-$ L(G) $-action on a von Neumann algebra $ N $, then for any $ x\in N $ and any $ k\in A(G) $, we have:
	\[(k\cdot x)\ot1_{\LT}\in\wsp\{(1_{N}\ot\lambda_{s})\delta(h\cdot k\cdot x):\ s\in G,\ h\in A(G) \}, \]
	where $ k\cdot x=(\mathrm{id }_{N}\ot k)(\delta(x)) $, for $ k\in A(G) $ and $ x\in N $.	
\end{remark}
Using this we can prove that the converse of Proposition \ref{p1} is true for the Hopf-von Neumann algebra $ (L(G),\delta_{G}) $.
\begin{cor}\label{cor6.2} Let $ (Y,\delta) $ be an $ L(G) $-comodule. Then, the following are equivalent:
	\begin{itemize}
		\item[(i)] $ Y=\wsp\{h\cdot y:\ h\in A(G),\ y\in Y \} $; 
		\item[(ii)] $ (Y,\delta) $ is non-degenerate;
		\item[(iii)] $ Y\vt \BLT=\wsp\{(1_{H}\ot b)\delta(y)(1_{H}\ot c):\ y\in Y,\ b,c\in \BLT \} $,
	\end{itemize}
	where $ h\cdot y=(\mathrm{id}_{Y}\ot h)(\delta(y)) $, for $ h\in A(G) $ and $ y\in Y $.		
\end{cor}
\begin{proof} The implication (iii)$ \implies $(i) can be proved in a similar manner as Proposition \ref{p1} and (ii)$ \implies $(iii) is obvious. Thus, it suffices to prove the implication (i)$ \implies $(ii).
	
	(i)$ \implies $(ii): Because of Remark \ref{rem1}, we may assume that $ (Y,\delta) $ is an $ L(G) $-subcomodule of a W*-$ L(G) $-comodule $ (N,\delta) $. Also, let us suppose that $ N $ acts on a Hilbert space $ K $. Since $ Y=\wsp\{h\cdot y:\ h\in A(G),\ y\in Y \} $, it follows  from \cite[Corollary II.1.5]{SVZ2} (see Remark \ref{cor6.1} above) that
	\[z\ot1_{\LT}\in\wsp\{(1_{K}\otimes b)\delta(y):\ b\in \BLT,\ y\in Y \}, \] 
	for any $ z\in Y $.	Therefore, for any $ z\in Y $ and $ c\in \BLT $, we get that
	\[ z\ot c=(1_{K}\ot c)(z\ot 1_{\LT})\in\wsp\{(1_{K}\otimes b)\delta(y):\ b\in \BLT,\ y\in Y \}, \] 
	because the multiplication in $ B(K)\vt\BLT $ is separately w*-continuous. Thus, we have that 
	\[Y\vt\BLT\sub\wsp\{(1_{K}\otimes b)\delta(y):\ b\in \BLT,\ y\in Y \}, \]
	which gives the desired equality since the reverse inclusion is trivial.
\end{proof}

\begin{remark}\label{re2}
	Observe that for any $ u,\ h\in A(G) $ and $ y\in L(G) $ we have:
	\begin{align*} 
	\la M_{u}(y), h\ra&=\la y, hu\ra\\
	&=\la\delta_{G}(y),h\ot u\ra\\
	&=\la(\mathrm{id}_{L(G)}\ot u)\circ\delta_{G}(y),h\ra,      
	\end{align*}
	therefore  \begin{equation}\label{eq1}
	M_{u}=(\mathrm{id}_{L(G)}\ot u)\circ\delta_{G},\quad \text{for all }  u\in A(G).
	\end{equation}
\end{remark}

\begin{pro}\label{p9} For a locally compact group $ G $ the following conditions are equivalent:
	\begin{itemize}	
		\item [(a)] $ G $ has the AP;
		\item[(b)] Every $ L(G) $-comodule is saturated;
		\item[(c)] For any  $ L(G) $-comodule $ (Y,\delta) $, any $ L(G) $-subcomodule $ Z $ of $ Y $  and any $ y\in Y $, we have that $ \delta(y)\in Z\ft L(G) $ implies $ y\in Z $;
		\item [(d)] For any  $ L(G) $-comodule $ (Y,\delta) $ and any $ y\in Y $, we have $ y\in \overline{A(G)\cdot y}^{\text{w*}} $;
		\item [(e)] There exists a net $ (u_{i})_{i\in I} $ in $ A(G) $ such that for any $ L(G) $-comodule $ (Y,\delta) $ and any $ y\in Y $ we have that $ u_{i}\cdot y\longrightarrow y $ ultraweakly;
		\item [(f)] Every $ L(G) $-comodule is non-degenerate.
	\end{itemize}	
\end{pro}
\begin{proof} The equivalence of (b), (c) and (d) follows immediately from Proposition \ref{p8}. The implication (f)$ \implies $(b) follows from Corollary \ref{c1} and the implication (d)$ \implies $(f) follows from Corollary \ref{cor6.2}. Also, (e)$ \implies $(d) is obvious. So, it suffices to prove the implications (a)$ \implies $(e) and (d)$ \implies $(a).
	
	(a)$ \implies $(e): Suppose that $ G $ has the AP. Then, by Theorem \ref{thm7.1}, there exists a net $ \{u_{i} \} $ in $ A(G)$ such that  \[(\mathrm{id}_{N}\ot M_{u_{i}})(z)\longrightarrow z\] ultraweakly, for any von Neumann algebra $ N $ and for all $ z\in N\vt L(G) $. 
	
	Let  $ (Y,\delta) $ be an $ L(G) $-comodule with $ Y $ being a w*-closed subspace of $ B(H) $ for some Hilbert space $ H $. First, we need to show that, for any $ u\in A(G) $, we have the following:
	\begin{equation}\label{eq2}
	(\mathrm{id}_{Y}\ot M_{u})\circ\delta=\delta\circ(\mathrm{id}_{Y}\ot u)\circ\delta.
	\end{equation}
	Indeed, it is not hard to see that 
	\[\delta\circ(\mathrm{id}_{Y}\ot u)=(\mathrm{id}_{Y}\ot \mathrm{id}_{L(G)}\ot u)\circ(\delta\ot \mathrm{id}_{L(G)}), \]
	thus, since $ (\delta\ot \mathrm{id}_{L(G)})\circ\delta=(\mathrm{id}_{Y}\ot\delta_{G}) \circ\delta $, we get:
	\begin{align*} \delta\circ(\mathrm{id}_{Y}\ot u)\circ\delta&=(\mathrm{id}_{Y}\ot \mathrm{id}_{L(G)}\ot u)\circ(\delta\ot \mathrm{id}_{L(G)})\circ\delta\\
	&=(\mathrm{id}_{Y}\ot \mathrm{id}_{L(G)}\ot u)\circ(\mathrm{id}_{Y}\ot\delta_{G}) \circ\delta\\
	&=\left[ \mathrm{id}_{Y}\ot\left(  (\mathrm{id}_{L(G)}\ot u)\circ\delta_{G}\right)  \right]\circ\delta\\
	&= (\mathrm{id}_{Y}\ot M_{u})\circ\delta,
	\end{align*}
	where the last equality follows from \eqref{eq1} (see Remark \ref{re2}).
	
	Now, for any $ y\in Y $ we have: \[(\mathrm{id}_{Y}\ot M_{u_{i}})(\delta(y))\longrightarrow \delta(y)\] ultraweakly, because $ \delta(Y)\sub Y\ft L(G)\sub B(H)\vt L(G) $. Thus, \eqref{eq2} implies that  \[\delta\circ(\mathrm{id}_{Y}\ot u_{i})\circ\delta(y)\longrightarrow \delta(y)\quad \text{ultraweakly.}\] On the other hand, $ \delta $ is a w*-continuous isometry, therefore it is a w*-w*- homeomorphism from $ Y $ onto $ \delta(Y) $ (by the Krein-Smulian theorem; see \cite{BLM} A.2.5) and thus $(\mathrm{id}_{Y}\ot u_{i})\circ\delta(y)\longrightarrow y$ ultraweakly, that is $ u_{i}\cdot y\longrightarrow y $ ultraweakly.	
	
	(d)$ \implies $(a): Assume that for any $ L(G) $-comodule $( Y,\delta) $ and any $ y\in Y $ we have that $ y\in \overline{A(G)\cdot y}^{\text{w*}} $. Let $ H $ be a Hilbert space. Thus taking  $ Y=B(H)\vt L(G)$ and $\delta=\Id_{B(H)}\ot\delta_{G} $ yields that for any $ y\in B(H)\vt L(G) $ there exists a net $ (u_{i}) $ in $ A(G) $ such that \[(\Id_{B(H)}\ot\Id_{L(G)}\ot u_{i})\circ(\Id_{B(H)}\ot\delta_{G})(y)\longrightarrow y \text{ ultraweakly}. \]
	Therefore, since 
	\begin{align*} 
	(\Id_{B(H)}\ot\Id_{L(G)}\ot u_{i})\circ(\Id_{B(H)}\ot\delta_{G})&= \mathrm{id}_{B(H)}\ot\left(  (\mathrm{id}_{L(G)}\ot u_{i})\circ\delta_{G}\right)\\
	&= \mathrm{id}_{B(H)}\ot M_{u_{i}},
	\end{align*}
	it follows that for any Hilbert space $ H $ and any $ y\in B(H)\vt L(G) $ there exists a net $ (u_{i}) $ in $ A(G) $ such that $ (\mathrm{id}_{B(H)}\ot M_{u_{i}})(y)\longrightarrow y $ ultraweakly.
	
	Now, consider a separable infinite dimensional Hilbert space $ K $ and let $ F=\{x_{1},\dots,x_{n}\} $ be a finite subset of $ B(K)\vt L(G) $. Then, $ x=x_{1}\oplus\dots\oplus x_{n} $ may be viewed as an element of $ B(K^{(n)})\vt L(G) $, where $ K^{(n)} $ is the direct sum of $ n $ copies of $ K $. Therefore, applying the above argument for $ K^{(n)} $ we get that there exists a net $ (u_{i}) $ in $ A(G) $ such that $ (\mathrm{id}_{B(K^{(n)})}\ot M_{u_{i}})(x)\longrightarrow x $ ultraweakly and thus it follows that $ (\mathrm{id}_{B(K)}\ot M_{u_{i}})(y)\longrightarrow y $ ultraweakly for all $ y\in F $. Therefore, if $ F $ is a finite subset of $ B(K)\vt L(G) $ and  $ \mathfrak{N} $ is an ultraweak neighborhood of $ 0 $, then there is an element $ u_{(F,\mathfrak{N})}\in A(G) $ such that 
	\[(\Id_{B(K)}\ot M_{u_{(F,\mathfrak{N})}})(y)\in y+\mathfrak{N},\quad\forall y\in F. \]
	So, the set of all pairs $ (F,\mathfrak{N}) $ becomes a directed set with the partial order defined by $ (F_{1},\mathfrak{N}_{1})\leq(F_{2},\mathfrak{N}_{2}) $ if $ F_{1}\sub F_{2} $ and $ \mathfrak{N}_{2}\leq\mathfrak{N}_{1} $ and it is clear that 
	\[(\Id_{B(K)}\ot M_{u_{(F,\mathfrak{N})}})(y)\longrightarrow y\ \text{ ultraweakly for all }y\in B(K)\vt L(G).\]
	Therefore, it follows from Theorem \ref{thm7.1} that $ G $ has the AP.	
\end{proof} 

\begin{remark}\label{rem6} According to Proposition \ref{p9}, if $ G $ satisfies the AP, then every $ L(G) $-comodule is saturated and non-degenerate.
	
On the other hand, if $ G $ does not satisfy the AP, then Proposition \ref{p9} guarantees the existence of $ L(G) $-comodules which are not saturated and the existence of $ L(G) $-comodules that are not non-degenerate. However, the author does not know any example of a group $ G $ without the AP such that there  exists a single $ L(G) $-comodule which is neither saturated nor non-degenerate. 

Furthermore, the next simple result suggests that such an example could never be isomorphic to the spatial or the Fubini crossed product of some $ \LI $-comodule.
\end{remark}

\begin{cor}\label{c10} For every $ \LI $-comodule $ (X,\alpha) $ the Fubini crossed product $ (X\fcr_{\alpha}G,\wh{\alpha}) $ is a saturated $ L(G) $-comodule and the spatial crossed product $ (X\scr_{\alpha}G,\wh{\alpha}) $ is a non-degenerate $ L(G) $-comodule.	
\end{cor} 
\begin{proof} Suppose that $ X $ is a w*-closed subspace of $ B(H) $ for some Hilbert space $ H $ and let $ K:= H\ot \LT $.
	
	First we show that $ (X\scr_{\alpha}G,\wh{\alpha}) $ is a non-degenerate $ L(G) $-comodule. We have that $ (X\scr_{\alpha}G)\vt\BLT $ is a $ \CIK\vt\BLT $-bimodule. Thus, since $ \wh{\alpha}(X\scr_{\alpha}G)\sub (X\scr_{\alpha}G)\ft L(G)\sub (X\scr_{\alpha}G)\vt\BLT $, we have the inclusion \[(X\scr_{\alpha}G)\vt\BLT\supseteq\wsp\{(1_{K}\ot b)\wh{\alpha}(y):\ b\in\BLT,\ y\in X\scr_{\alpha}G \}.\]	
	
	For the reverse inclusion, observe that for any $ s,t\in G $, $ x\in X $ and $b\in\BLT$, we have
	\begin{align*} 
	((1_{H}\ot\lambda_{s})\alpha(x)(1_{H}\ot\lambda_{t}))\ot b=&\ (1_{H}\ot1_{\LT}\ot b\lambda_{t}^{-1}\lambda_{s}^{-1})\\&\ (((1_{H}\ot\lambda_{s})\alpha(x)(1_{H}\ot\lambda_{t}))\ot\lambda_{st})\\
	=&\ (1_{H}\ot1_{\LT}\ot b\lambda_{st}^{-1})(1_{H}\ot\lambda_{s}\ot\lambda_{s})\\&\ (\alpha(x)\ot1_{\LT})(1_{H}\ot\lambda_{t}\ot\lambda_{t}) \\
	=&\ (1_{H}\ot1_{\LT}\ot b\lambda_{st}^{-1})(\Id_{B(H)}\ot\delta_{G})\\&\ ((1_{H}\ot\lambda_{s})\alpha(x)(1_{H}\ot\lambda_{t}))\\
	=&\ (1_{H}\ot1_{\LT}\ot b\lambda_{st}^{-1})\wh{\alpha}((1_{H}\ot\lambda_{s})\alpha(x)(1_{H}\ot\lambda_{t})),
	\end{align*} 
	since \[(\Id_{B(H)}\ot\delta_{G})(1_{H}\ot\lambda_{s})=1_{H}\ot\lambda_{s}\ot\lambda_{s}\text{ and }(\Id_{B(H)}\ot\delta_{G})(\alpha(x))=\alpha(x)\ot 1. \]
	Therefore, we get \[((1_{H}\ot\lambda_{s})\alpha(x)(1_{H}\ot\lambda_{t}))\ot b\in\wsp\{(1_{K}\ot c) \wh{\alpha}(y):\ c\in\BLT,\ y\in X\scr_{\alpha}G \}.\] 
	
	Since $ (X\scr_{\alpha}G)\vt\BLT $ is in the w*-closed linear span of the elements of the form $ ((1_{H}\ot\lambda_{s})\alpha(x)(1_{H}\ot\lambda_{t}))\ot b $, we obtain the desired inclusion and so $ (X\scr_{\alpha}G,\wh{\alpha}) $ is non-degenerate.
	
	For the Fubini crossed product, note that by Lemma \ref{l1} $ (X\vt\BLT,\Id_{X}\ot\delta_{G}) $ is saturated, because $ (\BLT,\delta_{G}) $ is saturated (see Remark \ref{rem7}). Thus, since the actions $ \wt{\alpha} $ and $ \Id_{X}\ot\delta_{G} $ on $ X\vt\BLT $ commute and $ X\fcr_{\alpha}G=\left(X\vt\BLT \right)^{\wt{\alpha}}  $ and $ \wh{\alpha}=(\Id_{X}\ot\delta_{G})|_{X\fcr_{\alpha}G } $, by Lemma \ref{l2}  it follows that $(X\fcr_{\alpha}G,\wh{\alpha} ) $ is saturated.		
\end{proof}  

\section{Takesaki-type duality for crossed products}\label{sec5}

Recall that every $ \LI $-comodule is non-degenerate and saturated (by Lemma \ref{c2}). Using the non-degeneracy we will obtain the Takesaki-duality for the spatial crossed product, i.e.
\[\left(X\scr_{\alpha}G \right)\lt_{\wh{\alpha}}G\simeq X\vt\BLT,  \]
whereas the saturation of $ (X,\alpha) $ yields the same for the Fubini crossed product, that is
\[\left(X\fcr_{\alpha}G \right)\lt_{\wh{\alpha}}G\simeq X\vt\BLT\]
(see Propositions \ref{p14} and \ref{p16} below).

The same ideas can be used to show that an $ L(G) $-comodule $ (Y,\delta) $ is non-degenerate if and only if $$ \left(Y\lt_{\delta}G \right)\scr_{\wh{\delta}}G\simeq Y\vt\BLT  ,$$ whereas $ (Y,\delta) $ is saturated if and only if $$ \left(Y\lt_{\delta}G \right)\fcr_{\wh{\delta}}G\simeq Y\vt\BLT  $$ (see Propositions \ref{p4} and \ref{p15}).

As a consequence we get our two main results. The first one (Theorem \ref{thm2}) states that for a fixed $ \LI $-comodule $ (X,\alpha) $ the equality $ X\fcr_{\alpha}G=X\scr_{\alpha}G $ holds if and only if $ (X\fcr_{\alpha}G,\wh{\alpha}) $ is non-degenerate if and only if $ (X\scr_{\alpha}G,\wh{\alpha}) $ is saturated. The second one (Theorem \ref{thm4}) states that the locally compact group $ G $ has the AP if and only if $ X\fcr_{\alpha}G=X\scr_{\alpha}G $ holds for any $ \LI $-comodule $ (X,\alpha) $.

\begin{remark}\label{rem7ii} It is known that if $ (Y,\delta) $ is a W*-$ L(G) $-comodule, then the double crossed product $\left(  (Y\lt_{\delta}G)\rtimes_{\wh{\delta}}G,\wh{\wh{\delta}} \right) $ is isomorphic to $ (Y\vt\BLT,\wt{\delta}) $ the isomorphism given by the map  $ \pi=(\mathrm{id}_{Y}\ot \mathrm{Ad}W_{\Lambda})\circ(\delta\ot \mathrm{id}_{\BLT})$ (see for example \cite[Theorem II.2.1]{SVZ2} or \cite[Chapter I, Theorem 2.7]{NaTa}). 
\end{remark}

We will prove that for an $ L(G) $-comodule $ (Y,\delta) $ (not necessarily a von Neumann algebra) the same map $ \pi $ maps the tensor product $ Y\vt\BLT $ onto the double spatial crossed product $ (Y\lt_{\delta}G)\scr_{\wh{\delta}}G $ if and only if $ (Y,\delta) $ is non-degenerate. Also, we show that $ \pi $ maps $ Y\vt \BLT $ onto the double Fubini crossed product $ (Y\lt_{\delta}G)\fcr_{\wh{\delta}}G $ if and only if $ (Y,\delta) $ is saturated.	

\begin{pro}\label{p4} Let $ (Y,\delta) $ be an $ L(G) $-comodule and consider the map  \[\pi=(\mathrm{id}_{Y}\ot \mathrm{Ad}W_{\Lambda})\circ(\delta\ot \mathrm{id}_{\BLT})\colon Y\vt\BLT\to Y\vt\BLT\vt\BLT.\] 
	Then, $ \pi $ satisfies
	\[\pi(Y\lt_{\delta}G)=\wh{\delta}(Y\lt_{\delta}G) \]
and the following conditions are equivalent:
	\begin{itemize}
		\item[(i)] $ (Y,\delta) $ is non-degenerate; 
		\item[(ii)] The map $ \pi $ is an $ L(G) $-comodule isomorphism from $ (Y\vt\BLT,\wt{\delta}) $ onto the double spatial crossed product $ ( (Y\lt_{\delta}G)\scr_{\alpha}G,\ \wh{\alpha} )  $, where $ \alpha=\wh{\delta}=(\mathrm{id}_{Y}\ot\beta_{G})|_{Y\lt_{\delta}G} $.
	\end{itemize}   	 
\end{pro}
\begin{proof} Suppose that $ Y $ is a w*-closed subspace of $ B(H) $ for some Hilbert space $ H $ and put $ X:=Y\lt_{\delta}G $.

We claim that, for any $ s,t\in G $, $ f,g\in\LI $ and $ y\in Y $, we have:
\begin{align} \label{(**)}
\pi((1\ot\rho_{t} f)\delta(y)(1\ot g\rho_{s}))=(1\ot1\ot\lambda_{t})\alpha((1\ot f)\delta(y)(1\ot g))(1\ot1\ot\lambda_{s}).
\end{align}
In order to prove \eqref{(**)}, first observe the following:
\begin{align*} 
W_{\Lambda}(1\ot \rho_{t}f)W_{\Lambda}^{*}&=W_{\Lambda}(1\ot \rho_{t})W_{\Lambda}^{*}W_{\Lambda}(1\ot f)W_{\Lambda}^{*}\\
&=(1\ot\Lambda)W_{G}(1\ot \rho_{t})W_{G}^{*}(1\ot\Lambda)W_{\Lambda}S(f\ot 1)SW_{\Lambda}^{*}\\&=(1\ot\Lambda)(1\ot \rho_{t})W_{G}W_{G}^{*}(1\ot\Lambda)W_{\Lambda}S\delta_{G}(f)SW_{\Lambda}^{*}\\
&=(1\ot\Lambda)(1\ot \rho_{t})(1\ot\Lambda)W_{\Lambda}SW_{G}^{*}(f\ot 1)W_{G}SW_{\Lambda}^{*}	\\
&=(1\ot\lambda_{t}) U_{G}^{*}(f\ot1) U_{G} \\
&= (1\ot\lambda_{t}) \beta_{G}(f)
\end{align*}
and similarly $ W_{\Lambda}(1\ot g\rho_{s})W_{\Lambda}^{*}=\beta_{G}(g)(1\ot \lambda_{s}) $. Also, we have \[\alpha(\delta(y))= (\Id\ot\beta_{G})(\delta(y))=\delta(y)\ot1.\]
Thus we get:	
\begin{align*} 
\pi((1\ot\rho_{t} f)\delta(y)(1\ot g\rho_{s}))=&(1\ot W_{\Lambda})(\delta\ot \mathrm{id})((1\ot\rho_{t}f)\delta(y)(1\ot g\rho_{s}))(1\ot W_{\Lambda}*)\\
=&(1\ot W_{\Lambda})(1\ot 1\ot\rho_{t}f)(\delta\ot \mathrm{id})(\delta(y))(1\ot g\rho_{s})(1\ot W_{\Lambda}^{*})\\
=&(1\ot W_{\Lambda})(1\ot 1\ot\rho_{t}f)(\mathrm{id}\ot\delta_{G})(\delta(y))(1\ot1\ot g\rho_{s})\\&(1\ot W_{\Lambda}^{*})\\
=&(1\ot W_{\Lambda})(1\ot 1\ot\rho_{t}f)(1\ot W_{G}^{*})(\delta(y)\ot 1)(1\ot W_{G})\\&(1\ot1\ot g\rho_{s})(1\ot W_{\Lambda}^{*})\\
=&(1\ot W_{\Lambda})(1\ot 1\ot\rho_{t}f)(1\ot W_{G}^{*})(1\ot1\ot \Lambda)(\delta(y)\ot 1)\\&(1\ot1\ot \Lambda)(1\ot W_{G})(1\ot1\ot g\rho_{s})(1\ot W_{\Lambda}^{*})\\	
=&(1\ot W_{\Lambda})(1\ot 1\ot\rho_{t}f)(1\ot W_{\Lambda}^{*})(\delta(y)\ot 1)(1\ot W_{\Lambda})\\&(1\ot1\ot g\rho_{s})(1\ot W_{\Lambda}^{*})\\
=&(1\ot 1\ot\lambda_{t})\left[ (\Id\ot\beta_{G})((1\ot f)\delta(y)(1\ot g))\right] (1\ot1\ot\lambda_{s})\\
=&(1\ot 1\ot\lambda_{t})\alpha((1\ot f)\delta(y)(1\ot g)) (1\ot1\ot\lambda_{s})	
\end{align*} 
and hence \eqref{(**)} is proved. The equality $ \pi(X)=\alpha(X) $ follows easily from \eqref{(**)}.
	
(i)$ \implies $(ii): Suppose that $ (Y,\delta) $ is non-degenerate. Since $$ \BLT=\wsp\{R(G)\LI\}=\wsp\{\LI R(G)\},$$ we have
	\begin{align}\label{(*)} 
	Y\vt\BLT&=\wsp\left\lbrace \left(\CI\vt\BLT\right)\delta(Y) \left(\CI\vt\BLT\right)\right\rbrace\\
	\nonumber&=\wsp\left\lbrace (1\ot\rho_{t}f)\delta(y)(1\ot g\rho_{s})):\ s,t\in G,\ f,g\in \LI,\ y\in Y\right\rbrace.	
	\end{align}
		
	 Clearly, the equality $ \pi(Y\vt\BLT)=X\scr_{\alpha}G $ follows from \eqref{(**)} and \eqref{(*)}. It remains to prove that $ \pi $ is an $ L(G) $-comodule isomorphism from $ (Y\vt\BLT,\wt{\delta}) $ onto $ ( X\scr_{\alpha}G,\ \wh{\alpha} )  $ . Since $ \pi $ is completely isometric and onto $ X\scr_{\alpha}G $, it suffices to prove that 
	\begin{equation} 
	\wh{\alpha}\circ\pi(x)=(\pi\ot \mathrm{id})\circ\wt{\delta}(x),\quad \forall x\in Y\vt\BLT.\label{eq3}
	\end{equation}
	Since $ (Y,\delta) $ is non-degenerate, we have
	\[Y\vt\BLT=\wsp\left\lbrace\left(\CI\vt\BLT\right)\delta(Y) \right\rbrace  \] 
	and thus we only have to verify \eqref{eq3} for  $ x=(1\ot\rho_{t}f)\delta(y) $, where $  t\in G $, $ f\in\LI $ and $ y\in Y $.  Indeed, this follows immediately from the calculations below: 
	\begin{align*} 
	\wh{\alpha}\circ\pi((1\ot\rho_{t}f)\delta(y))&=\wh{\alpha}((1\ot1\ot\lambda_{t})\alpha((1\ot f)\delta(y)))\\
	&=\left[ (1\ot1\ot\lambda_{t})\alpha((1\ot f)\delta(y))\right] \ot\lambda_{t}\\
	&=\left[ (1\ot1\ot\lambda_{t})\pi((1\ot f)\delta(y))\right] \ot\lambda_{t}\\
	&=\pi((1\ot\rho_{t}f)\delta(y))\ot\lambda_{t}\\
	&=(\pi\ot \mathrm{id})(\left[ (1\ot\rho_{t}f)\delta(y)\right] \ot\lambda_{t})
	\end{align*} 
	and on the other hand we have
	\begin{align*}  
	\wt{\delta}((1\ot\rho_{t}f)\delta(y))=&(\mathrm{id}\ot \mathrm{Ad}W_{G}\circ \sigma)\circ(\delta\ot \mathrm{id})((1\ot\rho_{t}f)\delta(y))\\
	=&(\mathrm{id}\ot \mathrm{Ad}W_{G}\circ \sigma)((1\ot1\ot\rho_{t}f)(\mathrm{id}\ot\delta_{G})(\delta(y)))\\
	=&(\mathrm{id}\ot \mathrm{Ad}W_{G})(1\ot\rho_{t}f\ot1)\left[ \delta(y)\ot1\right]\\ 
	=& (1\ot W_{G}(\rho_{t}\ot1)W_{G}^{*})(1\ot W_{G}(f\ot1)W_{G}^{*})(\delta(y)\ot1)\\
	=&(1\ot \rho_{t}\ot\lambda_{t})(1\ot f\ot1)(\delta(y)\ot1)\\
	=&[(1\ot\rho_{t}f)\delta(y)]\ot\lambda_{t}.
	\end{align*}
	
	(ii)$ \implies $(i): By Corollary \ref{c10}  we have that $(X\scr_{\alpha}G,\wh{\alpha}) $ is non-degenerate for any $ \LI $-comodule $ (X,\alpha) $. Therefore, $ ( (Y\dscr_{\delta}G)\scr_{\alpha}G,\ \wh{\alpha} )  $ is always non-degenerate and since it is isomorphic to $ (Y\vt\BLT,\wt{\delta}) $ (by assumption), it follows that $ (Y\vt\BLT,\wt{\delta}) $ is non-degenerate too. That is:
	\[Y\vt\BLT\vt\BLT=\wsp\{N\wt{\delta}(Y\vt\BLT)N \},\]
	where $ N:=\mathbb{C}1_{H}\vt\mathbb{C}1_{\LT}\vt\BLT $. Also put $ M:=\mathbb{C}1_{H}\vt\BLT$ $\vt\BLT  $. Thus, we get: 
	\begin{align*} 
	Y\vt&\BLT\vt\BLT=\\
	&=\wsp\{N(1_{H}\ot W_{G}S)(\delta\ot \mathrm{id}_{\BLT})(Y\vt\BLT)(1_{H}\ot SW_{G}^{*})N \}\\
	&\sub\wsp\{M(\delta(Y)\vt\BLT)M \}\\
	&\sub\left( \wsp\{(\CI\vt\BLT)\delta(Y)(\CI\vt\BLT)\}\right) \vt\BLT
	\end{align*}
	and therefore \[Y\vt\BLT=\wsp\{(\CI\vt\BLT)\delta(Y)(\CI\vt\BLT)\},\] which means that $ (Y,\delta) $ is non-degenerate (by Corollary \ref{cor6.2}).	
\end{proof}

\begin{remark}\label{rem7} It is known that every W*-$ L(G) $-comodule is saturated (see e.g. \cite[Proposition II.1.1]{SVZ2}). So, if $ Y $ is an $ L(G) $-subcomodule of a W*-$ L(G) $-comodule $ (N,\delta) $, then we have
	\begin{align*} 
	   \mathrm{Sat}(Y,\delta|_{Y})&=(Y\ft L(G))\cap \mathrm{Sat}(N,\delta)\\&=(Y\ft L(G))\cap\delta(N)
	\end{align*}
and thus $ (Y,\delta|_{Y}) $ is saturated if and only if $ \delta(Y)=(Y\ft L(G))\cap\delta(N) $.	The latter is equivalent to the condition $ \delta(Y)=(Y\vt\BLT)\cap\delta(N) $, since $ Y\ft L(G)=(Y\vt\BLT)\cap(N\vt L(G)) $ and $ \delta(N)\sub N\vt L(G) $.
\end{remark}

\begin{pro}\label{p15}  Let $ (Y,\delta) $ be an $ L(G) $-comodule and consider the map  $$ \pi=(\mathrm{id}_{Y}\ot \mathrm{Ad}W_{\Lambda})\circ(\delta\ot \mathrm{id}_{\BLT})\colon Y\vt\BLT\to Y\vt\BLT\vt\BLT $$ as in Proposition \ref{p4}. Then, the following are equivalent:
	\begin{itemize}
		\item[(i)] $ (Y,\delta) $ is saturated; 
		\item[(ii)] The map $ \pi $ is an $ L(G) $-comodule isomorphism from $ (Y\vt\BLT,\wt{\delta}) $ onto the double Fubini crossed product $ ( (Y\lt_{\delta}G)\fcr_{\alpha}G,\ \wh{\alpha} )  $, where $ \alpha=\wh{\delta}=(\mathrm{id}_{Y}\ot\beta_{G})|_{Y\lt_{\delta}G} $.
	\end{itemize} 
	
\end{pro}
\begin{proof} By Remark \ref{rem1} and Remark \ref{rem7iii} we may assume that $ (Y,\delta) $ is an $ L(G) $-subcomodule of some W*-$ L(G) $-comodule $ N $, i.e. $ N $ is a von Neumann algebra such that $ Y $ is a w*-closed subspace of $ N $ and $\delta $ extends to a W*-$ L(G) $-action on $ N $, which we still denote by $ \delta $ for simplicity. Also, the map $ \pi $ extends to the map $ (\mathrm{id}_{N}\ot \mathrm{Ad}W_{\Lambda})\circ(\delta\ot \mathrm{id}_{\BLT}) $ which gives the $ L(G) $-comodule isomorphism between $\left(  (N\lt_{\delta}G)\rtimes_{\wh{\delta}}G,\wh{\wh{\delta}} \right) $ and $ (N\vt\BLT,\wt{\delta}) $ (see Remark \ref{rem7ii}). It follows that $ \pi $ is an $ L(G) $-comodule monomorphism from $ (Y\vt\BLT,\wt{\delta}) $ into $ ( Y\vt \BLT\vt \BLT,\ \mathrm{id}_{Y}\ot\Id_{\BLT}\ot\delta_{G} )  $. 
	
Thus it suffices to show that $ \pi $ maps $ Y\vt\BLT $ onto $ (Y\lt_{\delta}G)\fcr_{\alpha}G$ if and only if $ (Y,\delta) $ is saturated.
	
First observe that 
\begin{align*} 
    Y\lt_{\delta}G&=\left( Y\vt\BLT\right)^{\wt{\delta}} \\
    &=\left( Y\vt\BLT\right)\cap\left( N\vt\BLT\right)^{\wt{\delta}}\\
    &= \left( Y\vt\BLT\right)\cap\left(N\lt_{\delta}G \right) 
\end{align*}	
and thus
\begin{align*} 
 \left(Y\lt_{\delta}G \right) \vt\BLT&=\left[\left( Y\vt\BLT\right)\cap\left(N\lt_{\delta}G \right) \right] \vt\BLT\\
 &=\left(Y\vt\BLT\vt\BLT \right) \cap\left[ \left( N\lt_{\delta}G\right)\vt\BLT \right]. 
\end{align*}
From the above equality and Remark \ref{rem7ii}  we get:
\begin{align*} 
      (Y\lt_{\delta}G)\fcr_{\alpha}G&=\left( \left(Y\lt_{\delta}G \right) \vt\BLT\right) ^{\wt{\alpha}}\\
      &=\left(Y\vt\BLT\vt\BLT \right)\cap \left[ \left( N\lt_{\delta}G\right)\vt\BLT \right]^{\wt{(\wh{\delta})}}\\
      &=\left(Y\vt\BLT\vt\BLT \right)\cap\left[ (N\lt_{\delta}G)\rtimes_{\wh{\delta}}G\right]\\
      &= \left(Y\vt\BLT\vt\BLT \right)\cap \pi\left(N\vt\BLT \right).
\end{align*}

Therefore, the equality \[\pi\left(Y\vt\BLT \right)=(Y\lt_{\delta}G)\fcr_{\alpha}G  \]
	is equivalent to \begin{equation}\label{4}
	\pi\left(Y\vt\BLT \right)=\left(Y\vt\BLT\vt\BLT \right)\cap \pi\left(N\vt\BLT \right).
	\end{equation}
Since  \[\pi\left(Y\vt\BLT \right)=(1\ot W_{\Lambda})(\delta(Y)\vt\BLT)(1\ot W_{\Lambda}^{*}) \]  and \[ \pi\left(N\vt\BLT \right)=(1\ot W_{\Lambda})(\delta(N)\vt\BLT)(1\ot W_{\Lambda}^{*}),  \] the equality \eqref{4} is equivalent to
\begin{align*} 
   \delta(Y)\vt\BLT&=\left[ Y\vt W_{\Lambda}\left(\BLT\vt\BLT \right)W_{\Lambda}^{*}  \right]\cap \left(\delta(N)\vt\BLT \right) \\
   &=\left(Y\vt\BLT\vt\BLT \right)\cap \left(\delta(N)\vt\BLT \right)\\
   &=\left[\left( Y\vt\BLT\right) \cap\delta(N) \right] \vt\BLT,
\end{align*}
or equivalently, \[\delta(Y)=\left( Y\vt\BLT\right) \cap\delta(N), \]	
which, by Remark \ref{rem7}, is true if and only if $ (Y,\delta) $ is saturated.
\end{proof}

The following corollary follows immediately from Proposition \ref{p9} and Propositions \ref{p4} and \ref{p15}.

\begin{cor}\label{c12} For a locally compact group $ G $, the following conditions are equivalent:
	\begin{itemize}
		\item [(1)] $ G $ has the AP;\\		
		\item [(2)] $ ( (Y\lt_{\delta}G)\fcr_{\alpha}G,\ \wh{\alpha} )\simeq(Y\vt\BLT,\wt{\alpha}) $ for any $ L(G) $-comodule $ (Y,\delta) $;\\		
		\item[(3)] $ ( (Y\lt_{\delta}G)\scr_{\alpha}G,\ \wh{\alpha} )\simeq (Y\vt\BLT,\wt{\alpha}) $ for any $ L(G) $-comodule $ (Y,\delta) $,
	\end{itemize}
where the above isomorphisms are given as in Propositions \ref{p4} and \ref{p15}.
\end{cor}

Thus, since the Takesaki-duality holds for crossed products of W*-$ L(G) $-comodules for any locally compact group $ G $ (as pointed out in Remark \ref{rem7ii}) and since there are groups without the AP (see \cite{LdS}), it becomes apparent that the generalization from the von Neumann algebraic setting to the setting of dual operator spaces is rather non-trivial.

Now, using that all $ \LI $-comodules are both non-degenerate and saturated, we prove that both the double Fubini crossed product $(X\fcr_{\alpha}G)\lt_{\widehat{\alpha}}G$ and the double spatial crossed product $ (X\scr_{\alpha}G)\lt_{\widehat{\alpha}}G $ of an $ \LI $-comodule $ (X,\alpha) $ are canonically isomorphic to the tensor product $ X\vt\BLT $ (see Propositions \ref{p14} and \ref{p16}). In fact, it follows that they are equal (Corollary \ref{c11}), even though $ X\scr_{\alpha}G $ and $ X\fcr_{\alpha}G $ may be different (recall Remark \ref{rem9}).

\begin{pro}\label{p14} Let $ (X,\alpha) $ be an $\LI $-comodule and put $$ \pi=(\mathrm{id}_{X}\ot \mathrm{Ad}V_{G})\circ(\alpha\ot \mathrm{id}_{\BLT})\colon X\vt\BLT\to X\vt\BLT\vt\BLT.$$ Then, the map $ \pi $ is an $ \LI $-comodule isomorphism from $ (X\vt\BLT,\wt{\alpha}) $ onto $ ( (X\scr_{\alpha}G)\lt_{\delta}G,\ \wh{\delta} )  $, where $ \delta=\wh{\alpha}=(\mathrm{id}_{X}\ot\delta_{G})|_{X\scr_{\alpha}G} $.
	In addition, $ \pi $ satisfies
	\[\pi(X\scr_{\alpha}G)=\delta(X\scr_{\alpha}G). \] 	
\end{pro}
\begin{proof} Suppose that $ X $ is a w*-closed subspace of $B(H)$ for some Hilbert space $ H $ and let $ K:= H\ot\LT $. Since $ (X,\alpha) $ is non-degenerate (by Lemma \ref{c2}) and \[ \BLT=\wsp\{\LI L(G) \}=\wsp\{L(G)\LI\} \] we have:
	\begin{align} 
	   X\vt\BLT=&\ \wsp\lbrace (\CI\vt \LI )(\CI\vt L(G) )\alpha(X)(\CI\vt L(G) )\nonumber\\&\ (\CI\vt \LI)    \rbrace \nonumber\\
	   =&\ \wsp\lbrace (\CI\vt \LI)(X\scr_{\alpha}G) (\CI\vt \LI) \rbrace. \label{eq7}
	\end{align}
On the other hand, by the definition of the crossed product, we have:
\begin{align} 
       (X\scr_{\alpha}G)\lt_{\delta}G=&\ \wsp\lbrace (\CIK\vt\LI)\delta(X\scr_{\alpha}G)(\CIK\vt\LI \rbrace \label{eq8}
\end{align}	
Since the map $ \pi $ is clearly a w*-continuous complete isometry, by the equalities \eqref{eq7} and \eqref{eq8}, in order to prove that $ \pi $ is an $ \LI $-comodule isomorphism onto $ (X\scr_{\alpha}G)\lt_{\delta}G $, it suffices to verify the following conditions:
\begin{equation}\label{eq9}
     \pi(X\scr_{\alpha}G)=\delta(X\scr_{\alpha}G),
\end{equation} 
\begin{equation}\label{eq10}  
      \pi((1\ot f)y)=(1\ot1\ot f)\pi(y),\quad \text{ for }f\in\LI \text{ and } y\in X\vt\BLT,
\end{equation}
and
\begin{equation}\label{eq11}
   \wh{\delta}\circ\pi=(\pi\ot\Id)\circ\wt{\alpha}
\end{equation}
For any $ x\in X $ and $ s,t\in G $ we have:
\begin{align*}  
    \pi((1\ot\lambda_{s})\alpha(x)(1\ot\lambda_{t}))=&\ (\Id\ot\Ad V_{G})( (\alpha\ot \Id)((1\ot\lambda_{s})\alpha(x)(1\ot\lambda_{t})))\\
    =&\ (\Id\ot\Ad V_{G})( (1\ot 1\ot\lambda_{s})(\alpha\ot\Id)(\alpha(x)) (1\ot 1\ot\lambda_{t}) )\\
    =&\ (\Id\ot\Ad V_{G})( (1\ot 1\ot\lambda_{s})(\Id\ot\alpha_{G})(\alpha(x)) (1\ot 1\ot\lambda_{t}) )\\
    =&\ (1\ot V_{G}) (1\ot 1\ot\lambda_{s})(\Id\ot\alpha_{G})(\alpha(x)) (1\ot 1\ot\lambda_{t})(1\ot V_{G}^{*})\\
    =&\ (1\ot V_{G})(1\ot 1\ot\lambda_{s})(1\ot V_{G}^{*})(\alpha(x)\ot1)(1\ot V_{G})\\&\ (1\ot 1\ot\lambda_{t})(1\ot V_{G}^{*})\\
    =&\ (1\ot SW^{*}_{G}S)(1\ot 1\ot\lambda_{s})(1\ot SW_{G}S)(\alpha(x)\ot1)\\&\ (1\ot SW^{*}_{G}S)(1\ot 1\ot\lambda_{t})(1\ot SW_{G}S)\\
    =&\ (1\ot SW^{*}_{G})(1\ot \lambda_{s}\ot 1)(1\ot W_{G}S)(\alpha(x)\ot1)(1\ot SW^{*}_{G})\\&\ (1\ot\lambda_{t}\ot 1)(1\ot W_{G}S)\\
    =&\ (1\ot S)(1\ot \lambda_{s}\ot \lambda_{s})(1\ot S)(\alpha(x)\ot1)(1\ot S)\\&\ (1\ot\lambda_{t}\ot \lambda_{t})(1\ot S)\\
    =&\ (1\ot \lambda_{s}\ot \lambda_{s})(\alpha(x)\ot1)(1\ot\lambda_{t}\ot \lambda_{t})\\
    =&\ (\Id\ot\delta_{G})((1\ot\lambda_{s})\alpha(x)(1\ot\lambda_{t}))
\end{align*}
and thus the equality \eqref{eq9} is proved.

On the other hand, for any $ f,g\in\LI $ and $ y\in X\vt \BLT $, we have:
\begin{align*} 
    \pi((1\ot f)y(1\ot g))=&\ (\Id\ot\Ad V_{G}) \circ(\alpha\ot\Id)((1\ot f)y(1\ot g)) \\
    =&\ (\Id\ot\Ad V_{G}) ( (1\ot 1\ot f)(\alpha\ot\Id)(y) (1\ot 1\ot g))\\
    =&\ (1\ot V_{G})(1\ot 1\ot f)(1\ot V_{G}^{*})\pi(y)(1\ot V_{G})(1\ot 1\ot g)(1\ot V_{G}^{*})\\
    =&\ (1\ot 1\ot f)\pi(y)(1\ot 1\ot g),
\end{align*}
because $ V_{G}\in L(G)\vt\LI $ and therefore it commutes with $ 1\ot f $ and $ 1\ot g $. Hence we have proved \eqref{eq10}.

Since $ (X,\alpha) $ is non-degenerate, we have
\[X\vt \BLT=\wsp\lbrace (\CI\vt\LI)(\CI\vt L(G))\alpha(X) \rbrace \]
and thus we only need to verify \eqref{eq11} for elements of the form $ y=(1\ot f)(1\ot\lambda_{s})\alpha(x) $, where $ f\in \LI $, $ s\in G $ and $ x\in X $. Indeed, we have

\begin{align*} 
   \wt{\alpha}((1\ot f)(1\ot\lambda_{s})\alpha(x))=&\ (\Id\ot\Ad U_{G}^{*})\circ(\Id\ot\sigma)\left[  (\alpha\ot\Id)(( 1\ot f\lambda_{s})\alpha(x)) \right] \\
   =&\ (\Id\ot\Ad U_{G}^{*})\circ(\Id\ot\sigma)\left[  (1\ot 1\ot f\lambda_{s})(\Id\ot\alpha_{G})(\alpha(x)) \right] \\
   =&\ (\Id\ot\Ad U_{G}^{*})\left[  (1\ot f\lambda_{s}\ot 1)(\Id\ot\Ad U_{G})(\alpha(x)\ot 1) \right] \\
   =&\ (\Id\ot\Ad U_{G}^{*})(1\ot f\ot 1)(\Id\ot\Ad U_{G}^{*})(1 \ot \lambda_{s} \ot 1)\\&\ (\alpha(x)\ot 1)\\
   =&\ (1\ot\beta_{G}(f))(1\ot\lambda_{s}\ot 1)(\alpha(x)\ot 1)
\end{align*}
and therefore we get 
\begin{align*}
    (\pi\ot\Id)\circ\wt{\alpha}((1\ot f)(1\ot\lambda_{s})\alpha(x))=&\ (\pi\ot\Id)((1\ot\beta_{G}(f))(1\ot\lambda_{s}\ot 1)(\alpha(x)\ot 1))\\
    =&\ (1\ot 1\ot\beta_{G}(f))\left(\pi((1\ot\lambda_{s})\alpha(x))\ot 1 \right) \\
    =&\  (1\ot 1\ot\beta_{G}(f))(\delta((1\ot\lambda_{s})\alpha(x))\ot 1)\\
    =&\ (1\ot 1\ot\beta_{G}(f))\wh{\delta}(\delta((1\ot\lambda_{s})\alpha(x)))\\
    =&\ (\Id\ot\Id\ot\beta_{G})( (1\ot 1\ot f)\delta((1\ot\lambda_{s})\alpha(x)) )\\
    =&\ (\Id\ot\Id\ot\beta_{G})( (1\ot 1\ot f)\pi((1\ot\lambda_{s})\alpha(x)) )\\
    =&\ \wh{\delta}\circ\pi((1\ot f)(1\ot\lambda_{s})\alpha(x)).
\end{align*}
\end{proof}

Since every $ \LI $-comodule is saturated (by Lemma \ref{c2}), Proposition \ref{p16} below can be proved in a similar manner as Proposition \ref{p15} using the Takesaki-duality theorem for W*-$ \LI $-comodules (see for example \cite[Chapter I, Theorem 2.5]{NaTa})  and so we omit its proof. Alternatively, Proposition \ref{p16} could be derived from \cite[Proposition 5.7]{Ha1} since one can easily verify that the notion of $ G $-completeness introduced by Hamana in \cite{Ha1} is equivalent to saturation for $ \LI $-comodules.

\begin{pro} \label{p16} Let $ (X,\alpha) $ be an $\LI $-comodule and let  $ \pi=(\mathrm{id}_{X}\ot \mathrm{Ad}V_{G})\circ(\alpha\ot \mathrm{id}_{\BLT})\colon X\vt\BLT\to X\vt\BLT\vt\BLT $. Then, the map $ \pi $ is an $ \LI $-comodule isomorphism from $ (X\vt\BLT,\wt{\alpha}) $ onto $ ( (X\fcr_{\alpha}G)\lt_{\delta}G,\ \wh{\delta} )  $, where $ \delta=\wh{\alpha}=(\mathrm{id}_{X}\ot\delta_{G})|_{X\fcr_{\alpha}G} $.
	In addition, $ \pi $ satisfies
	\[\pi\left( X\fcr_{\alpha}G\right) =\delta\left( X\fcr_{\alpha}G\right) . \] 		
\end{pro}

The next simple corollary essentially states that, for an $ \LI $-comodule $ (X,\alpha) $, the saturation space of the spatial crossed product $ X\scr_{\alpha}G $ is isomorphic to the Fubini crossed product $ X\fcr_{\alpha}G $.

\begin{cor}\label{c11} For any $ \LI $-comodule $ (X,\alpha) $ we have:
	\begin{itemize}
		\item [(i)] $ (X\fcr_{\alpha}G)\lt_{\wh{\alpha}}G=(X\scr_{\alpha}G)\lt_{\wh{\alpha}}G $;
		\item [(ii)] $ \mathrm{Sat}(X\scr_{\alpha}G,\wh{\alpha})=\mathrm{Sat}\left(X\fcr_{\alpha}G,\wh{\alpha}\right)=\wh{\alpha}\left(X\fcr_{\alpha}G\right)  $.
	\end{itemize}	
\end{cor}
\begin{proof} Statement (i) is an obvious consequence of Propositions \ref{p14} and \ref{p16}. So, we only need to show (ii). Indeed, we have 
	\begin{align*}
	\mathrm{Sat}(X\scr_{\alpha}G,\wh{\alpha})&=\left((X\scr_{\alpha}G)\lt_{\wh{\alpha}}G \right)^{\wh{\wh{\alpha}}} \\
	&=\left((X\fcr_{\alpha}G)\lt_{\wh{\alpha}} G\right)^{\wh{\wh{\alpha}}} \\
	&=\mathrm{Sat}(X\fcr_{\alpha}G,\wh{\alpha})\\
	&=\wh{\alpha}\left(X\fcr_{\alpha}G\right),
	\end{align*}
where both the first equality and the third equality follow from Proposition \ref{p11}, the second one follows from statement (i) and the fourth equality is because $ \left(X\fcr_{\alpha}G, \wh{\alpha}\right) $ is a saturated $ L(G) $-comodule by Corollary \ref{c10}.	
\end{proof}

\begin{thm}\label{thm2} Let $ (X,\alpha) $ be an $ \LI $-comodule. The following are equivalent:
	\begin{itemize}
		\item [(a)] $ X\fcr_{\alpha}G=X\scr_{\alpha}G $;
		\item [(b)] $ \left(X\fcr_{\alpha}G,\wh{\alpha}\right)  $ is a non-degenerate $ L(G) $-comodule;
		\item [(c)] $ (X\scr_{\alpha}G,\wh{\alpha} ) $ is a saturated $ L(G) $-comodule.
	\end{itemize} 
\end{thm}
\begin{proof} The implications (a)$ \implies $(b) and (a)$ \implies $(c) are immediate from Corollary \ref{c10}. Hence it remains to prove the implications (c)$ \implies $(a) and (b)$\implies$(a).
	
(c)$ \implies $(a): Let  $ (X\scr_{\alpha}G,\wh{\alpha} ) $ be saturated, that is by definition $ \mathrm{Sat}(X\scr_{\alpha}G,\wh{\alpha})=\wh{\alpha}(X\scr_{\alpha}G) $. So, Corollary \ref{c11} (ii) yields that $ \wh{\alpha}\left(X\fcr_{\alpha}G\right) =\wh{\alpha}\left(X\scr_{\alpha}G\right)  $ and therefore $ X\fcr_{\alpha}G=X\scr_{\alpha}G $, since $ \wh{\alpha} $ is isometric.

(b)$ \implies $(a): Let $ \left(X\fcr_{\alpha}G,\wh{\alpha}\right) $ be non-degenerate. Then, by Proposition \ref{p1}, it follows that 
\[X\fcr_{\alpha}G=\wsp\left\lbrace A(G)\cdot\left(X\fcr_{\alpha}G\right)\right\rbrace.  \]
On the other hand, from Corollary \ref{c11} (ii) we have
\[\wh{\alpha}\left(X\fcr_{\alpha}G \right)= \mathrm{Sat}(X\scr_{\alpha}G,\wh{\alpha})\sub \left(X\scr_{\alpha}G \right) \ft L(G) \] 
and thus, by the definition of the Fubini tensor product, we get
\[h\cdot y=(\Id_{X\fcr_{\alpha}G}\ot h)(\wh{\alpha}(y))\in X\scr_{\alpha}G,\quad\forall h\in A(G),\ \forall y\in X\fcr_{\alpha}G, \]	
that is $ A(G)\cdot\left(X\fcr_{\alpha}G\right)\sub X\scr_{\alpha}G  $. Therefore \[X\fcr_{\alpha}G=\wsp\left\lbrace A(G)\cdot\left(X\fcr_{\alpha}G\right)\right\rbrace\sub X\scr_{\alpha}G,\] which yields the equality $ X\scr_{\alpha}G=X\fcr_{\alpha}G $, since $ X\scr_{\alpha}G\sub X\fcr_{\alpha}G $ is always true (see Proposition \ref{pro2.13}).
\end{proof}

\begin{remark} The equivalence between (a) and (b) in Theorem \ref{thm2} above is exactly \cite[Theorem 3.17]{DA} and therefore Theorem \ref{thm2} improves \cite[Theorem 3.17]{DA}.  It should be noted that the proof of \cite[Theorem 3.17]{DA} is based on an idea similar to that of Proposition \ref{p10}.	
\end{remark}

\begin{lem}\label{l4} For any $ L(G) $-comodule $ (Y,\delta) $, $ (\mathrm{Sat}(Y,\delta),\Id_{Y}\ot\delta_{G}) $ is a saturated $ L(G) $-comodule.	
\end{lem}
\begin{proof} Take an $ L(G) $-comodule $ (Y,\delta) $ and suppose that $ Y $ is a w*-closed subspace of $ B(K) $ for some Hilbert space $ K $. Since $ (Y,\delta) \simeq (\delta(Y),\Id_{Y}\ot\delta_{G}) $ (see Remark \ref{rem1}), it follows that $( \mathrm{Sat}(Y,\delta),\Id_{Y}\ot\delta_{G}) \simeq \left( \mathrm{Sat}\left( \delta(Y),\Id_{Y}\ot\delta_{G} \right),\Id_{\delta(Y)}\ot\delta_{G}\right)   $ (by Remark \ref{rem7iii}) and thus it suffices to prove that $ \left( \mathrm{Sat}\left( \delta(Y),\Id_{Y}\ot\delta_{G} \right),\Id_{\delta(Y)}\ot\delta_{G}\right)   $ is saturated (again by Remark \ref{rem7iii}).	Indeed, by Remark \ref{rem7}, we have:
\[\mathrm{Sat}\left( \delta(Y),\Id_{Y}\ot\delta_{G} \right)=(\Id_{B(K)}\ot\delta_{G})(B(K)\vt L(G))\cap(\delta(Y)\ft L(G)). \]

On the other hand, $ (\Id_{B(K)}\ot\delta_{G})(B(K)\vt L(G))$ is a W*-$ L(G) $-subcomodule of $ (B(K)\vt L(G)\vt L(G),\Id_{B(K)}\ot\Id_{L(G)}\ot\delta_{G}) $ and hence saturated and $\delta(Y)\ft L(G) $ is a saturated $ L(G) $-subcomodule of $ (B(K)\vt L(G)\vt L(G),\Id_{B(K)}\ot\Id_{L(G)}\ot\delta_{G}) $ by Lemma \ref{l1} since $ (L(G),\delta_{G}) $ is saturated (as a W*-$ L(G) $-comodule). Therefore, the desired conclusion follows from the fact that the intersection of saturated subcomodules is clearly saturated too.
\end{proof}		

We are now in position to prove Theorem \ref{thm4} below which complements Proposition \ref{p9}. It should be noted that the implication (i)$ \implies $(iii) in Theorem \ref{thm4} is essentially \cite[Corollary 4.8]{CN} and thus Theorem \ref{thm4} is a stronger result. It is very interesting that Crann and Neufang use a quite different approach to this based on their Fej\'{e}r-type theorem for elements in crossed products of groups with the AP (see \cite{CN} for more details) which seems to be independent of the Takesaki-duality. However, from the proof of Theorem \ref{thm4}, it becomes clear that the Takesaki-duality not only helps us to prove the converse of \cite[Corollary 4.8]{CN}, but also it yields an alternative proof of \cite[Corollary 4.8]{CN} itself via Theorem \ref{thm2}.

\begin{thm} \label{thm4} For a locally compact group $ G $ the following conditions are equivalent:
	\begin{itemize}
		\item [(i)] $ G $ has the AP;
		\item [(ii)] $ (Y\lt_{\delta}G)\fcr_{\wh{\delta}}G=(Y\lt_{\delta}G)\scr_{\wh{\delta}}G $, for any $ L(G) $-comodule $ (Y,\delta) $;
		\item [(iii)] $ X\fcr_{\alpha}G=X\scr_{\alpha}G $, for any $ \LI $-comodule $ (X,\alpha) $.
	\end{itemize}	
\end{thm}
\begin{proof} The implication (iii)$ \implies $(ii) is obvious. Also, the implication (i)$ \implies $(iii) is a direct consequence of Proposition \ref{p9} and Theorem \ref{thm2}. Therefore, it suffices to prove that (ii) implies (i).
	
Suppose that condition (ii) is true. Then, every saturated $ L(G) $-comodule is non-degenerate. Indeed, if $ (Y,\delta) $ is a saturated $ L(G) $-comodule, then Proposition \ref{p15} yields that 
\[(Y\lt_{\delta}G)\fcr_{\wh{\delta}}G=\pi(Y\vt\BLT), \]
where $ \pi=(\mathrm{id}_{Y}\ot \mathrm{Ad}W_{\Lambda})\circ(\delta\ot \mathrm{id}_{\BLT}) $. Thus, since $ (Y\lt_{\delta}G)\fcr_{\wh{\delta}}G=(Y\lt_{\delta}G)\scr_{\wh{\delta}}G $	by assumption, it follows that $ (Y\lt_{\delta}G)\scr_{\wh{\delta}}G=\pi(Y\vt\BLT) $, which in turn implies that $ (Y,\delta) $ is non-degenerate (by Proposition \ref{p4}).

Hence, if condition (ii) holds and $ (Z,\varepsilon) $ is an arbitrary $ L(G) $-comodule, then the $ L(G) $-comodule $ (\mathrm{Sat}(Z,\varepsilon),\Id_{Z}\ot\delta_{G}) $ is non-degenerate, because $ (\mathrm{Sat}(Z,\varepsilon),\Id_{Z}\ot\delta_{G}) $ is saturated (by Lemma \ref{l4}) and every saturated $ L(G) $-comodule is non-degenerate by the previous argument. Therefore, by Proposition \ref{p2}, we get that $ (Z,\varepsilon) $ is non-degenerate.

So, it follows that condition (ii) implies that every $ L(G) $-comodule is\mbox{ } non-degenerate and thus Proposition \ref{p9} yields that (ii) implies (i).	 
\end{proof}

\section{Realization of masa bimodules and harmonic operators as crossed products}\label{sec6} 

Let $ J $ be a closed ideal of the Fourier algebra $ A(G) $ and let $ J^{\perp} $ be its annihilator in $ L(G) $. Then, it is easy to see that $ J^{\perp} $ is an $ L(G) $-subcomodule of $ (L(G),\delta_{G}) $, that is $ \delta_{G}(J^{\perp})\sub J^{\perp}\ft L(G) $. In fact, every $ L(G) $-subcomodule of $ (L(G),\delta_{G}) $ arises this way by taking the preannihilator in  $ A(G) $.

Following Anoussis, Katavolos and Todorov (see \cite{AKT1}), one can define two w*-closed $ \LI $-bimodules inside $ \BLT $, namely:
\[\B:=\wsp\{fTg:\ T\in J^{\perp},\ f,g\in \LI \} \]
and
\[\Sa:=\bigcap\{\ker S_{u}:\ u\in J \},\]
where 
\[ S_{u}(T)=(\Id\ot u)\circ\delta_{G}(T)=u\cdot T,\text{ for any }  T\in\BLT \text{ and } u\in A(G).\]
That is, $ S_{u} $ is the (unique) completely bounded w*-continuous $ \LI $-bimodule map on $ \BLT $ which maps $ f\lambda_{s} $ to $ u(s)f\lambda_{s} $ for all $ s\in G $ and $ f\in \LI $.

Note that the original definition of $ \Sa $ (denoted $ (\mathrm{Sat}J)^{\perp} $ by the authors) in \cite{AKT1} is different from ours. However, one can easily verify the equivalence of the two definitions using \cite[Proposition 3.1]{AKT1} and the comments and definitions of \cite[Section 2]{AKT1}. 

According to the next theorem, which is one of the main results of \cite{AKT1}, the above $ \LI $-bimodules are equal. 

\begin{thm}\cite[Theorem 3.2]{AKT1} \label{thmAKT} For any closed ideal $ J $ of the Fourier algebra $ A(G) $ it holds that $ \Sa=\B $. 
	
\end{thm} 

\begin{remark}\label{rem10} If $ \mu $ is a probability measure on the locally compact group $ G $, a Borel function $ f\colon g\to \mathbb{C} $ such that
	\[f(s)=\int_{G} f(t^{-1}s)d\mu(t),\quad\text{ for all } s\in G, \]
is called $ \mu $-harmonic \cite{CL}. On the other hand, considering the dual pair $ (A(G),L(G)) $ as the non-commutative replacement of the dual pair $ (L^{1}(\wh{G}),L^{\infty}(\wh{G})) $ for abelian $ G $, Chu and Lau introduced the non-commutative analogue of harmonic functions in terms of functionals on $ A(G) $. In particular, if $ \sigma $ is a function in the Fourier-Stieltjes algebra $ B(G) $, then the space $ \cl{H}_{\sigma} $ of $ \sigma $-harmonic functionals on $ A(G) $ is defined \cite{CL} as the annihilator in $ L(G) $ of the closed ideal of $ A(G) $ generated by the elements of the form $ \sigma h-h, $ for $ h\in A(G) $.

Neufang and Runde \cite{NR} introduced the $ \sigma $-harmonic operators $ \wt{\cl{H}}_{\sigma} $ through a suitable extension of the action of $ \sigma $ on $ \BLT $. They proved that $ \wt{\cl{H}}_{\sigma}$ is the von Neumann algebra generated by $\LI$ and $\cl{H}_{\sigma}$.

The above result of Neufang and Runde was later generalized by Anoussis, Katavolos and Todorov in \cite{AKT2}. More precisely, for a family $ \Sigma $ of completely bounded multipliers of $ A(G) $, they defined the space $ \cl{H}_{\Sigma} $ of jointly $ \Sigma $-harmonic functionals and the space $ \wt{\cl{H}}_{\Sigma} $ of jointly $ \Sigma $-harmonic operators. They proved that $ \wt{\cl{H}}_{\Sigma} $ is the w*-closed $ \LI $-bimodule generated by $ \cl{H}_{\Sigma} $ by using \cite[Theorem 3.2]{AKT1} (that is Theorem \ref{thmAKT} above) and the fact that $ \wt{\cl{H}}_{\Sigma} $ is of the form $ \B $ for the ideal $ J=\Sigma A(G) $ of $ A(G) $ generated by the elements $ \sigma u $ for $ \sigma\in\Sigma $ and $ u\in A(G) $. 

Furthermore, from  \cite[Proposition 2.15]{AKT2} it follows that the w*-closed subspaces of $ \BLT $ of the form $ \B $ for some closed ideal $ J $ of $ A(G) $ are exactly those of the form $ \wt{\cl{H}}_{\Sigma}  $ for some family $ \Sigma\sub M_{0}A(G) $. 
\end{remark}

In the following, we are going to show that $ \B $ or equivalently $ \wt{\cl{H}}_{\Sigma}  $ can be realized as the crossed product $ J^{\perp}\lt_{\delta_{G}}G $ of the $ L(G) $-comodule $ (J^{\perp},\delta_{G}) $ (see Proposition \ref{p12}). In other words, we show that harmonic operators are examples of crossed products of $ L(G) $-comodules, which are not necessarily von Neumann algebras.

This complements the identification of analogous $ L(G) $-bimodules introduced in \cite{AKT} with spatial and Fubini crossed products by $ G $, obtained in \cite[Proposition 5.1]{DA}.

Furthermore, thanks to Proposition \ref{p12} we obtain an alternative proof of \cite[Theorem 3.2]{AKT1} by a direct application of Theorem \ref{thm1}.

\begin{pro}\label{p12}   The normal *-monomorphism  \[\Phi\colon\BLT\to\BLT\vt\BLT\]  defined by
	\[\Phi(x)=W_{\Lambda}^{*}\beta_{G}(T)W_{\Lambda}=\Ad[(U_{G}W_{\Lambda})^{*}](x\ot 1),\quad T\in \BLT, \]
	is an $ \LI $-comodule monomorphism with respect to the $ \LI $-actions $ \beta_{G} $ and $ \Id_{\BLT}\ot\beta_{G} $ on $ \BLT $ and $ \BLT\vt \BLT $ respectively. Also, if $ J$ is a closed ideal of $ A(G) $, then $ \Phi $ maps $ \B $ onto $ J^{\perp}\dscr_{\delta_{G}}G $ and $ \Sa $ onto	$ J^{\perp}\dfcr_{\delta_{G}}G $.	
	Therefore, $ \B=\Sa $.	
\end{pro}
\begin{proof}Recalling that $ U_{G}W_{\Lambda}=W_{G}S $, we have:
	\begin{align*} 
	\Phi(\lambda_{s})&=W_{\Lambda}^{*}U_{G}^{*}(\lambda_{s}\ot1)U_{G}W_{\Lambda}\\
	&=SW_{G}^{*}(\lambda_{s}\ot1)W_{G}S\\
	&=S\delta_{G}(\lambda_{s})S\\
	&=\delta_{G}(\lambda_{s}),
	\end{align*}
	for any $ s\in G $.
	
	On the other hand, since $ \delta_{G}(f)=f\ot 1 $ for any $ f\in \LI $ we get:
	\begin{align*} 
	\Phi(f)&=W_{\Lambda}^{*}U_{G}^{*}(f\ot1)U_{G}W_{\Lambda}\\&=SW_{G}^{*}(f\ot1)W_{G}S\\&=S(f\ot1)S\\&=1\ot f,
	\end{align*}
	for all $ f\in\LI $. Therefore, $ \Phi(\BLT)=L(G)\dscr_{\delta_{G}}G=L(G)\dfcr_{\delta_{G}}G $, because $ \BLT $ is the w*-closed linear span of $ \LI L(G) $.
	
	Also, from the above calculations it is clear that maps the generators
	of $ \B $ onto those of $ J^{\perp}\dscr_{\delta_{G}}G $, and so $ \Phi(\B)=J^{\perp}\dscr_{\delta_{G}}G $.
	
	Since $ \BLT=\wsp\{\LI L(G) \} $, in order to prove that $ \Phi $ is an $ \LI $-comodule monomorphism with respect to the $ \LI $-actions $ \beta_{G} $ and $ \Id_{\BLT}\ot\beta_{G} $ it suffices to verify the equality \[(\Id_{\BLT}\ot\beta_{G})\circ\Phi(x)=(\Phi\ot \Id_{\LI})\circ\beta_{G}(x) \] for $ x=\lambda_{s} $, $ s\in G $, and for $ x=f\in\LI $. Indeed, for $ s\in G $ and $ f\in \LI $, we have:
	\begin{align*}  
	(\Id_{\BLT}\ot\beta_{G})\circ\Phi(\lambda_{s})&=(\Id_{\BLT}\ot\beta_{G})(\lambda_{s}\ot\lambda_{s})\\&=\lambda_{s}\ot\lambda_{s}\ot1\\&=\Phi(\lambda_{s})\ot1\\&=(\Phi\ot\Id)(\lambda_{s}\ot1)\\&=(\Phi\ot\Id)(\beta_{G}(\lambda_{s}))
	\end{align*}
	On the other hand, since $ \Phi(g)=1\ot g $ for all $ g\in\LI $ it follows that $ (\Phi\ot \Id)(y)=1\ot y $ for any $y\in\LI\vt\LI $. Therefore, when $ f\in \LI $ we get: 
	\begin{align*} 
	(\Id_{\BLT}\ot\beta_{G})\circ\Phi(f)&=(\Id_{\BLT}\ot\beta_{G})(1\ot f)\\&=1\ot\beta_{G}(f)\\&=(\Phi\ot \Id_{\LI})(\beta_{G}(f)),
	\end{align*}
	because $ \beta_{G}(f)\in \LI\vt\LI $.
	
	It remains to show that $ \Phi(\Sa)=J^{\perp}\dfcr_{\delta_{G}}G $. To this end, we first prove  the following:
	\begin{equation}\label{eq6} 
	S_{u}=(u\ot\Id_{\BLT})\circ\Phi,\quad\text{for all } u\in A(G).
	\end{equation}
	Indeed, if $ s\in G $ and $ f\in\LI $, then we have:
	\begin{align*} 
	S_{u}(f\lambda_{s})&=u(s)f\lambda_{s}\\&=(u\ot\Id)(\lambda_{s}\ot(f\lambda_{s}))\\&=(u\ot\Id)((1\ot f)(\lambda_{s}\ot\lambda_{s}))\\&=(u\ot\Id)(\Phi(f)\Phi(\lambda_{s}))\\&=(u\ot\Id)(\Phi(f\lambda_{s}))
	\end{align*}
	and hence \eqref{eq6} follows because $ \BLT=\wsp\{\LI L(G)\} $.
	
	 Therefore, we get: 
	\begin{align*} 
	J^{\perp}\dfcr_{\delta_{G}}G&=\left( J^{\perp}\vt\BLT \right)^{\wt{ \delta_{G} } }\\
	&=\left( L(G)\vt\BLT \right)^{\wt{ \delta_{G} } } \cap \left( J^{\perp}\vt\BLT \right)\\
	&= \left( L(G)\dfcr_{\delta_{G}}G \right) \cap \left( J^{\perp}\vt\BLT \right)\\
	&=\Phi(\BLT)\cap\left( J^{\perp}\vt\BLT \right)\\
	&=\left\lbrace T\in\Phi(\BLT):\ (\Id\ot\omega)(T)\in J^{\perp},\ \forall\omega\in\BLT_{*} \right\rbrace\\
	&=\left\lbrace T\in\Phi(\BLT):\ \la(\Id\ot\omega)(T),u\ra=0,\ \forall\omega\in\BLT_{*},\ \forall u\in J \right\rbrace\\
	&=\left\lbrace T\in\Phi(\BLT):\ \la(u\ot \Id)(T),\omega\ra=0,\ \forall\omega\in\BLT_{*},\ \forall u\in J \right\rbrace\\
	&=\left\lbrace T\in\Phi(\BLT):\ (u\ot \Id)(T)=0,\ \forall u\in J \right\rbrace\\
	&=\Phi\left( \bigcap\{\ker S_{u}:\ u\in J \} \right), 
	\end{align*}
	where the last equality follows from \eqref{eq6}. Thus $ J^{\perp}\dfcr_{\delta_{G}}G=\Phi(\Sa) $. 
	
	Finally, since $ J^{\perp}\dfcr_{\delta_{G}}G=J^{\perp}\scr_{\delta_{G}}G $ (by Theorem \ref{thm1}), it follows that $ \B=\Sa $.	
\end{proof} 

\begin{remark}\label{re1} Let $ J $ be a closed ideal of $ A(G) $. By Proposition \ref{p12} we get that $ \Sa $ is an $ \LI $-subcomodule of $ (\BLT,\beta_{G}) $. Observe that 
	\begin{align*} 
	\Sa^{\beta_{G}}&=\BLT^{\beta_{G}}\cap\Sa\\
	&=L(G)\cap\Sa
	\end{align*} 
	and thus the isomorphism $ \Phi\colon(\Sa,\beta_{G})\to(J^{\perp}\dfcr_{\delta_{G}}G,\wh{\delta_{G}}) $  maps the intersection $ L(G)\cap\Sa $ onto  $ \left(J^{\perp}\dfcr_{\delta_{G}}G \right)^{\wh{\delta_{G}}}=\mathrm{Sat}(J^{\perp},\delta_{G})  $ (see Proposition \ref{p11}). Also, $ \Phi(J^{\perp})=\delta_{G}(J^{\perp}) $.
\end{remark} 
Note that Remark \ref{re1} in combination with Theorem \ref{thmAKT} yields the following:

\begin{cor}\label{c8} For any closed ideal $ J $ of $ A(G) $ the following are equivalent:
	\begin{itemize}
		\item [(i)] $ (J^{\perp},\delta_{G}) $  is saturated;
		\item [(ii)] $ L(G)\cap\B=J^{\perp}  $.
	\end{itemize}	
\end{cor} 

In \cite{AKT1} Anoussis, Katavolos and Todorov proved that if $ A(G) $ admits an approximate unit (not necessarily bounded), then
\[L(G)\cap\B=J^{\perp}\] for any closed ideal $ J$ of $ A(G) $ \cite[Lemma 4.5]{AKT1}. They asked whether the same conclusion holds for an arbitrary group $ G $. Clearly, from Corollary \ref{c8}, this question is equivalent to asking whether every $ L(G) $-subcomodule of $ (L(G),\delta_{G}) $ is saturated. Using this point of view, we prove below (Proposition \ref{p7}) that a condition which is slightly weaker than the existence of a (possibly unbounded) approximate unit in $ A(G) $ is necessary and sufficient. This improves \cite[Lemma 4.5]{AKT1}.

\begin{defin} Let $ G $ be a locally compact group. Following \cite[Remark 5.1.8 (2)]{Kan}, we say that $ G $ has \textit{Ditkin's property at infinity} (or property $ D_{\infty} $ for short), if 
	\[u\in\overline{A(G)u}^{||\cdot||},\quad\forall u\in A(G). \]
	Also, following \cite{Ey}, we say that an element $ x\in L(G) $ satisfies \textit{condition (H)} if
	\[x\in \overline{A(G)\cdot x}^{\text{w*}}. \]
\end{defin}

Although the equivalence between statements (a) to (c) in the next result is already known (see e.g. \cite{Ey}), we have included its proof for the sake of completeness. 

\begin{pro}\label{p7} Let $ G $ be a locally compact group. Then, the following conditions are equivalent:
	\begin{itemize}
		\item[(a)] $ G $ has property $ D_{\infty} $. 
		\item[(b)] Every $ x\in L(G) $ satisfies condition (H).  
		\item[(c)] For any $ x\in L(G) $ and $ h\in A(G) $, if $ h\cdot x=0 $, then $ \la x,\ h\ra=0 $.
		\item[(d)] For any $ L(G) $-subcomodule $ Y $ of $ (L(G),\delta_{G}) $ and any $ x\in L(G) $ we have
		\[\delta_{G}(x)\in Y\ft L(G)\iff x\in Y. \] 
		\item[(e)] Every $ L(G) $-subcomodule of $ (L(G),\delta_{G}) $ is saturated.
		\item[(f)] Every $ L(G) $-subcomodule of $ (L(G),\delta_{G}) $ is non-degenerate. 
		\item[(g)] For every closed ideal $ J $ of $ A(G) $, we have $ L(G)\cap\B=J^{\perp}$. 
     \end{itemize}	
\end{pro} 	 
\begin{proof} (b)$ \implies $(a): Suppose that every element in $ L(G) $ satisfies condition (H) and that  there exists $ u\in A(G) $, such that $ u\notin\overline{A(G)u}^{||\cdot||} $. Then, there exists $ x\in L(G) $, such that $ \la x,u\ra\neq0 $ and $ \la x,vu\ra=0 $, for all $ v\in A(G) $. This means that $  \la v\cdot x,u\ra=0 $, for all $ v\in A(G) $ and since $ x $ satisfies condition (H) it is implied that $ \la x,u\ra=0 $, a contradiction.  	
	
	(c)$ \implies $ (b): Suppose that for any $ x\in L(G) $ and $ h\in A(G) $, $ h\cdot x=0 $ implies that $ \la x,\ h\ra=0 $. If there exists an  $ x\in L(G) $, such that $ x\notin \overline{A(G)\cdot x}^{\text{w*}}$, then there must be an $ h\in A(G) $, such that $ \la x,\ h\ra\neq0 $	and $ \la u\cdot x,h\ra=0 $, for any $ u\in A(G) $. But $ \la u\cdot x,h\ra=\la  x,hu\ra =\la  x,uh\ra=\la h\cdot x,u\ra$, therefore we get that $ \la h\cdot x,u\ra=0 $, for all $ u\in A(G) $ and thus $ h\cdot x=0 $, which implies that $ \la x, h\ra =0$, by hypothesis. Hence, we have a contradiction.
	
	(a)$ \implies $ (c): Assume that $ G $ has $ D_{\infty} $ and there exist $ x\in L(G) $ and $ h\in A(G) $, such that $ h\cdot x=0 $ and $ \la x,h\ra\neq0 $. Then, $ \la h\cdot x,u\ra=0 $, for all $ u\in A(G) $, that is $  \la  x,uh\ra=0 $, for all $ u\in A(G) $. But, since  $ G $ has $ D_{\infty} $, we have that there is a net $ (u_{i}) $ in $ A(G) $, such that $ u_{i}h\longrightarrow h $. Therefore, $ \la x,h\ra=\lim \la x,u_{i}h\ra=0 $, which is a contradiction. 
	
	(e)$ \implies $(d): Let $ Y $ be an $ L(G) $-subcomodule of $ (L(G),\delta_{G}) $ and let $ x\in L(G) $, with $ \delta_{G}(x)\in Y\ft L(G) $. Then, by the coassociativity of $ \delta_{G} $, we have that $ (\delta_{G}\ot \mathrm{id}_{L(G)})(\delta_{G}(x))=(\mathrm{id}_{Y}\ot\delta_{G}) (\delta_{G}(x)) $. Thus, $ \delta_{G}(x)\in \mathrm{Sat}(Y,\delta_{G})=\delta_{G}(Y) $, since $ Y $ is saturated. Therefore, $ x\in Y $, because $ \delta_{G} $ is isometric.
	
	(d)$ \implies $(b): Take an $ x\in L(G) $ and put $ Y:=\overline{A(G)\cdot x}^{\text{w*}} $. Then, $ Y $ is clearly a subcomodule of $ (L(G),\delta_{G}) $ (because it is an $ A(G) $-submodule) and $ \delta_{G}(x)\in Y\ft L(G) $. Indeed, if not, then there must be $ h,\ u\in A(G) $, such that $ \la y,u\ra=0 $, for all $ y\in Y $, and $ \la \delta_{G}(x),u\ot h \ra\neq0$. But $ \la \delta_{G}(x),u\ot h \ra\neq0$ implies that   $ \la h\cdot x,u \ra\neq0$, while $ u $ annihilates $ Y $ and $ h\cdot x\in Y $ by definition, which is a contradiction. Therefore, $ \delta_{G}(x)\in Y\ft L(G) $ and (d) implies that $ x\in Y $.
	
	(b)$ \implies $(f): This follows immediately from Corollary \ref{cor6.2}.
	
	(f)$ \implies $(e): Suppose that every $ L(G) $-subcomodule of $ (L(G),\delta_{G}) $ is non-degenerate. Let $ Y $ be an $ L(G) $-subcomodule of $ L(G) $. If we put
	\[Y_{1}:=\{x\in L(G):\ A(G)\cdot x\sub Y \}, \]
	then clearly $ Y_{1} $ is an $ L(G) $-subcomodule of $ L(G) $ which contains $ Y $. Furthermore, it is clear by the definition of $ Y_{1} $ that
	\[\delta_{G}(Y_{1})=(Y\ft L(G))\cap \delta_{G}(L(G))= \mathrm{Sat}(Y,\delta_{G}). \]
	 Since $ Y_{1} $ is non-degenerate by assumption, we get  that \[Y_{1}=\wsp\{A(G)\cdot Y_{1}\}\sub Y \] and therefore $ Y=Y_{1} $, that is $ Y $ is saturated because $ \delta_{G}(Y_{1})=\mathrm{Sat}(Y,\delta_{G}) $.
	 
	 (e)$ \iff $(g): This follows from Corollary \ref{c8} since the map $ J\mapsto J^{\perp} $ is clearly a bijection between the set of all closed ideals of $ A(G) $ and the set of all $ L(G) $-subcomodules of $ (L(G),\delta_{G}) $.
\end{proof}	 
 
\begin{remark} The author does not know whether $ A(G) $ has property $ D_{\infty} $ for all locally compact groups $ G $. On the other hand, Propositions \ref{p9} and \ref{p7} imply that if $ G $ has the AP, then $ A(G) $ has property $ D_{\infty} $ (see also \cite[Proposition 1.19]{HK} and the authors' comment after that), but whether the converse is true or not remains unknown. 
	
Also, Proposition \ref{p7} implies that a closed ideal $ J $ of the Fourier algebra $ A(G) $ can be recovered from $ \B $, i.e. the map $ J\mapsto \B $ is one to one, at least when $ G $ has property $ D_{\infty} $. It is unknown whether this map is one to one for an arbitrary locally compact group $ G $.   	 
\end{remark} 
	
\section*{Acknowledgments}
The author is really grateful to Aristides Katavolos and Michael Anoussis for helpful suggestions and discussions as well as to Matthias Neufang for stimulating discussions during his visit to the National and Kapodistrian University of Athens in September 2019. The research work was supported by the Hellenic Foundation for Research and Innovation (HFRI) and the General Secretariat for Research and Technology (GSRT), under the HFRI PhD Fellowship Grant (GA. no. 74159/2017).

\end{document}